\documentclass[10pt,a4paper]{article}
\usepackage[latin1]{inputenc}
\usepackage[margin=3cm]{geometry}
\usepackage{amsmath}
\usepackage{amsfonts}
\usepackage{stmaryrd}
\usepackage{amssymb}
\usepackage{multirow}
\usepackage{amsthm}
\usepackage{enumerate}
\usepackage[table]{xcolor}
\usepackage{hhline}
\usepackage{makecell}
\usepackage{tikz}
\usepackage{ifthen}
\usepackage{subfiles}
\usepackage{accents}
\usepackage{mathrsfs}
\newenvironment{tproof}{\begin{proof}[\textbf{Proof}]}{\end{proof}}
\newlength{\dhatheight}

\newcommand{\Pmir}[1]{P_{#1}^{\mathrm{mir}}}

\newcommand{\tj}{\tilde{j}}

\newcommand{\car}{\cap^{\mathrm{ent}}}
\newcommand{\eI}{\mathrm{I}}
\newcommand{\Mmir}[1]{M_{#1}^{\mathrm{mir}}}
\newcommand{\Umir}[1]{U_{#1}^{\mathrm{mir}}}
\newcommand{\Omsd}[1]{\mathcal{O}_{\mathrm{sd}}\left(#1\right)}
\newcommand{\RRR}{\mathcal{R}}

\newcommand{\tjwa}{\tj_{p,w}}
\newcommand{\tjw}{\tj_{p,w_{\mathrm{o}}}}
\newcommand{\wop}{w_{\mathrm{o}}}
\newcommand{\wlo}{w.l.o.g }

\newcommand{\tjh}{\tj_{\mathrm{hor}}}
\newcommand{\tjhp}{\tj_{\mathrm{hor}}^\prime}
\newcommand{\tjhpp}{\tj_{\mathrm{hor}}^{\prime\prime}}
\newcommand{\sd}{\mathrm{d}}

\newcommand{\quasi}[1]{\mathrm{quasi}(#1)}

\newcommand{\SSS}{\mathcal{S}}
\newcommand{\XXX}{\mathcal{X}}
\newcommand{\ih}{\mathrm{ih}}

\newcommand{\Aut}[1]{\mathrm{Aut}_{#1}}

\newcommand{\Lie}[1]{\mathrm{Lie}(#1)}
\newcommand{\ossa}{2\mathrm{a}1\mathrm{d}}

\newcommand{\tlarrow}{\leftarrowtriangle}
\newcommand{\mirsl}[1]{{P^{\mathrm{mir,as},0}_{#1}}}

\newcommand{\mult}[2]{\mathrm{mult}(#1,#2)}

\newcommand{\Stab}[2]{\mathrm{Stab}_{#1}(#2)}
\newcommand{\DDDD}[1][]{\mathcal{D}_{#1}} 

\newcommand{\kkk}{\mathfrak{k}}

\newcommand{\AAA}[1][]{\mathcal{A}_{\kkk#1}}
\newcommand{\AAnk}[1][]{\mathcal{A}_{#1}}
\newcommand{\Prink}[1][]{\mathcal{C}_{#1}}

\newcommand{\BBnk}[1][]{\mathcal{B}_{#1}}
\newcommand{\VVV}{\mathcal{V}}

\newcommand{\UUU}[1]{\mathcal{N}_{#1}}

\newcommand{\s}{\setminus}

\newcommand{\Pri}[1][]{\mathcal{C}_{\kkk#1}}

\newcommand{\Set}[1]{\mathrm{Set}(#1)}

\newcommand{\XX}{\mathrm{X}}

\newcommand{\FF}{\tilde{\mathrm{F}}}

\newcommand{\F}{\mathrm{F}}

\newcommand{\YY}{\mathrm{Y}}

\newcommand{\WW}{\mathrm{W} }

\newcommand{\Z}{\mathrm{Z}}

\newcommand{\I}[1][]{\mathrm{I}^{#1}}

\newcommand{\sXi}{\mathrm{\Xi}^{\mathrm{st}}}

\newcommand{\Om}{\mathcal{O}}

\newcommand{\Omm}[1]{\mathcal{O}_{\mathbf{f}}\left(#1\right)}

\newcommand{\row}[1]{\text{row}(#1)}

\newcommand{\col}[1]{\text{col}(#1)}

\newcommand{\bb}{\bullet}

\newcommand{\x}{\bb}

\newcommand{\eu}{\text{ue}}
\newcommand{\A}{\mathbb{A}_{\kkk}}

\newcommand{\co}{\text{c}}
\newcommand{\e}{\text{e}}
\newcommand{\gray}{gray!30}

\newcommand{\grid}[1]{\draw (0,0)[step=1,\gray]grid(#1,#1);}

\newcommand{\di}{6}
\newcommand{\mm}[2]{(#2,\di-#1)}

\newcommand{\ci}[4][\di]{\ifthenelse{\equal{#4}{a}}{\draw (#2-0.5,#1+0.5-#3)[gray!70,fill=gray!70]circle(0.17);}{\ifthenelse{\equal{#4}{b}}{\fill (#2-0.5,#1+0.5-#3)circle(0.2);}{\ifthenelse{\equal{#4}{c}}{\draw (#2-0.5,#1+0.5-#3)[line width=1pt,gray!70]circle(0.2);}{}} }\ifthenelse{\equal{#4}{d}}{\draw (#2-0.5,#1+0.5-#3)[line width=1pt]circle (0.2);}{\ifthenelse{\equal{#4}{e}}{\draw[line width=1pt,gray!70](#2-0.7,#1+0.3-#3)--(#2-0.3,#1+0.3-#3)--(#2-0.5,#1+0.8-#3)--(#2-0.7,#1+0.3-#3);}{\ifthenelse{\equal{#4}{f}}{\draw[line width=1pt](#2-0.7,#1+0.3-#3)--(#2-0.3,#1+0.3-#3)--(#2-0.5,#1+0.8-#3)--(#2-0.7,#1+0.3-#3); }}}} 

\usetikzlibrary{patterns}

\newcommand{\fir}[2][1]{\draw(0,#1)node[anchor=west,fill=white,inner sep=2pt]{$#2$};}

\newcounter{row}
\newenvironment{ti}[1][0.6]{\begin{tikzpicture}[scale=#1]\setcounter{row}{1}}{\end{tikzpicture}}
\newcommand{\f}[2]{\draw (\value{row}-0.5,\di+0.5-\value{row})node{#1};\foreach \x/\y in {#2} {\ci{\x}{\value{row}}{\y}}\stepcounter{row}}

\newcommand{\rif}{\overset{\mathrm{f}}{\rightarrow}}
\newcommand{\riff}[1]{\overset{\mathrm{f},#1}{\rightarrow}}
\newcommand{\rifff}[2]{\overset{\mathrm{f},#2,#1}{\rightarrow}}
\newcommand{\chh}[1][]{}

\newcommand{\Di}[1]{\mathrm{dim}\left(#1\right)} 

\newcommand{\X}[2][]{\mathbf{X}_{#1}(#2)}

\newcommand{\vsp}{\vspace{1,5mm}}
\newtheorem{theorem}{Theorem}[section]
\newtheorem{expectation}[theorem]{Expectation}
\newtheorem{maincor}[theorem]{Main corollary}
\newtheorem{Nilor}[theorem]{Nilpotent orbit dimension formula}

\newtheorem{prop}[theorem]{Proposition}

\newtheorem{lem}[theorem]{Lemma}

\newtheorem{ob}[theorem]{Observation} 
\newtheorem{Ecor}[theorem]{Exchange corollary}

\theoremstyle{definition}

\newtheorem{example}[theorem]{Example}

\newtheorem{axbx}[theorem]{Definition}
\newenvironment{defi}
  {\pushQED{\qed}\axbx}
  {\popQED\endaxbx}
  \newtheorem{axbxu}[theorem]{Some standard definitions and conventions}
  \newenvironment{stadefi}
  {\pushQED{\qed}\axbxu}
  {\popQED\endaxbxu}
  \newtheorem{axbxc}[theorem]{Proposition-Definition}
  \newenvironment{pdefi}
  {\pushQED{\qed}\axbxc}
  {\popQED\endaxbxc}

 \newtheorem{axbxxx}[theorem]{Remark}  
   \newenvironment{remark}
     {\pushQED{\qed}\axbxxx}
     {\popQED\endaxbxxx}

\theoremstyle{remark}

\DeclareFontFamily{U}{mathb}{\hyphenchar\font45}
\DeclareFontShape{U}{mathb}{m}{n}{
      <5> <6> <7> <8> <9> <10> gen * mathb
      <10.95> mathb10 <12> <14.4> <17.28> <20.74> <24.88> mathb12
      }{}
\DeclareSymbolFont{mathb}{U}{mathb}{m}{n}

\DeclareMathSymbol{\act}{3}{mathb}{'374}
\DeclareMathSymbol{\ssquare}{3}{mathb}{"05}
\newtheoremstyle{subdefi}{1pt}{1pt}{}{}{\bfseries}{.}{.5em}{}
\theoremstyle{subdefi}
\newtheorem*{axbxxxt}{Subdefinition}  
   \newenvironment{subdef*}
     {\pushQED{\qed}\axbxxxt}
     {\popQED\endaxbxxxt}

\begin{document}
\title{On Unfoldings of Some Integrals of Automorphic Functions on General Linear Groups}
\date{}
\author{Eleftherios Tsiokos}
\maketitle
\begin{abstract} We use results about Fourier coefficients appearing in \cite{Tsiokos2} (and some more obtained here), to obtain information for certain among the integrals of the form $$I=\int_{GL_n(\kkk)Z_n(\A)\s GL_n(\A)}\varphi(g)\phi(g)\F(E)(\tj(g))dg$$ where: $\A$ is the adele ring of a number field $\kkk$; $\varphi$ is a $GL_n(\A)$-cuspidal automorphic form; $\phi$ is a $GL_n(\A)$-automorphic function (even the trivial for some results); $E$ is a $GL_{N}(\A)$-automorphic form for a multiple $N$ of $n$; $\F(E)$ is a Fourier coefficient of $E$ for certain choices of additive functions $\F$ in a set $\BBnk[N]$ which we defined in $\cite{Tsiokos2};$ $\tj$ is a diagonal embedding of $GL_n$ in $GL_N$;  of course $\tj(GL_n)\in\Stab{GL_N}{\F}$; and $Z_n$ is the center of $GL_n$.
	 
\end{abstract}\tableofcontents \section{Integrals of automorphic functions}Results from \cite{Tsiokos2} and simple refinements of them that we use are given in the appendix. Unless otherwise specified, we retain the meaning of the notations and definitions in \cite{Tsiokos2}; the main exception is in relation to the mirabolic group in the definition below. In some occasions we recall or reference  notations from \cite{Tsiokos2}, but we soon and frequently use them without mention (even in their first occurrence here); in particular, notations about $\AAnk$-trees are freely used in the present paper: in many proofs and in the appendix.   

Relations of the present paper to the literature are given in Remark \ref{rlit} and Definition \ref{dequ}.

We choose an integer $n\geq 2$. Whenever $N$ appears, it is assumed to be a positive multiple of $n$.  By ``AF" we mean additive function. For $\F$ being an ``AF" we denote by $D_{\F}$ its domain. Let $H$ be an algebraic group containing $D_\F$ as an algebraic subgroup. For $\gamma\in H$ we define the AF $\gamma\F$ given by $\gamma\F(u)=\F(\gamma^{-1}u\gamma)$ for all $u\in D_{\gamma\F}:=\gamma D_\F\gamma^{-1}$. We also define $\Stab{H}{\F}=\{\gamma\in H:\gamma\F=\F\}.$
\begin{defi}[$\Pmir{n}$, $\Mmir{n}$, $\Umir{n}$]\label{dmir} We denote by $\Pmir{n}$ the $GL_n$-parabolic subgroup with Levi isomorphic to $GL_{n-1}\times GL_1$ so that $GL_1$ appears in the lower right corner. The unipotent radical (resp. the Levi) of $\Pmir{n}$ is denoted by $\Umir{n}$ (resp. $\Mmir{n}$). \end{defi}

\begin{stadefi}\label{stadefilab}Unless otherwise specified, for a product of general linear groups: the standard choice is assumed for root data (including parabolic and Levi subgroups); root data, rows, columns, and entries,  are with respect to the biggest such group that is mentioned. The previous sentence also holds for $\ossa$-groups, which we recall in Definition \ref{ossagroup}, but to be clear, it does not affect statements for diagonally embedded general linear groups. The Weyl group of $GL_n$ is denoted by $W_n$ and its elements are identified with permutation matrices.

In case not otherwise specified, an index variable is assumed to take all the values for which the statement containing it has a precise meaning\footnote{For example, due to this rule, in (ii) in the first definition below, we do not need to mention that $i$ takes as values: $1,...,t$. However, there are some occasions in which we mention (for emphasis) the values obtained by this rule.}. 

We denote by $U_n$ the set of upper triangular unipotent matrices of $GL_n$, and by $U_{n,(i,j)}$ or  $U_{(i,j)}$, the root group of $GL_n$ which is nontrivial on the $(i,j)$ (matrix) entry.

The transpose of a matrix $A$ is denoted by $A^t$.

We identify $\Mmir{N}$ with $GL_{N-1}\times GL_1$  as needed. For example (by recalling the concept ``$\Omm{\F}$" from \ref{AOF}) as a result of this identification the following statement is \textbf{wrong}:  $\Omm{\F}=\emptyset$ for all $\F$ with $D_\F\subseteq\Mmir{n}$.\end{stadefi} 

 The AFs $\F$ as in the abstract that we consider, belong to a set $\RRR_{n,N}$ defined below which  easily\footnote{Easy in the sense that it is an easy refinement of a well known argument, that is, we proceed as in the Whittaker expansion of $GL_N(\A)$-cusp forms, except that we deal with blocks and the process can become increasingly diagonal.}  turns out to be a subset of $\BBnk[N]$ (Part (i) in Lemma \ref{Blem}). The meaning of $\BBnk[N]$ is recalled in Definition \ref{defBBB}, where we also define extensions of $\BBnk[N]$ which we need (for Lemma \ref{Blem} and the proof of Proposition \ref{prop1}). Except for the last statement for the integrals $I$ in the abstract (which is Proposition \ref{Sprop}) we restrict to cases in which one among the $GL_N$- parabolic subgroups from which $E$ is induced is $\Pmir{N}$. We first encounter integrals $I$ as in the abstract, in Proposition \ref{ftprop}, where we only unfold until $I$ is expressed as a (finite) sum of integrals over a factorizable domain of Fourier coefficients. The AFs corresponding to these Fourier coefficients are first studied in Lemma \ref{Blem} after Part (i), where they are proved to belong to certain among the previously mentioned extensions of $\BBnk[...]$. Then in Proposition \ref{prop1} we return to the integrals $I$, and the information in the appendix together with Lemma \ref{Blem} are essential. Finally (the already mentioned) Proposition \ref{Sprop}   only adresses cases of $I$ in which $\F$ is the trivial AF; also, the proof of Proposition \ref{Sprop} uses the appendix to a lesser extent than Proposition \ref{prop1}, and by restricting the proofs of Propositions \ref{prop1} and \ref{Sprop} to the intersection of these two propositions, the proofs obtained have differences (see Remark \ref{avmain}). 
 
 As it happens in many cases\footnote{Cases different from these include many of the often called ``new way integrals".}  in the Rankin-Selberg method, in Propositions \ref{prop1} and \ref{Sprop}, a dimension equation formulated by D. Ginzburg---after we extend the set of integrals on which it has been formulated--- (see Definition \ref{dequ}) is satisfied, or is not satisfied but  in ways involving a modest extra effort at most.

\begin{defi}[$\RRR_{n,N}$, $P_i$, $Q_i$]\label{mainafs}	Consider an AF $\F$, and subgroups $P_i,Q_i,V_i$ of $GL_N$ for $1\leq i\leq t$ (for a positive integer $t$), and also $Q_{t+1}$, all of them defined over $\kkk$ (which is a number field fixed throughout the paper), as follows:\begin{itemize}
		\item[(i)] $D_\F=V_1...V_t$. 
		\item[(ii)] $V_i$ is the maximal abelian unipotent normal algebraic subgroup of $P_i$ which: normalizes the unipotent radical $U_{Q_i}$ of $Q_i$, and $V_iU_{Q_i}$ is the unipotent radical of $P_i$ (hence, if $P_i=Q_i$, then $V_i$ is the unipotent radical of $P_i$).

		\item[(iii)] $\F|_{V_i}$ is in the open orbit of the action by conjugation of $P_i$  on the AFs with domain $V_i$.
		\item[(iv)] $P_1$ is a maximal $GL_N$-parabolic subgroup and $Q_1=P_1$. Let $Q_{i+1}$ be defined so that we have the semidirect product of algebraic subgroups: $\Stab{P_i}{\F|_{V_i}}=V_i\rtimes Q_{i+1}$. We require that: $Q_{i}$ is equal to a general linear group (that is, any isomorphic copy of $GL_x$ (for some $x$)) or equal to a parabolic subgroup of a general linear group; $Q_{t+1}$ is  a general linear group; $P_{i}$ is equal to $Q_{i}$ or equal to a maximal parabolic subgroup of $Q_{i}$; and $P_i$ is not a general linear group. To avoid confusion, the usual assumption about varying indices (mentioned in \ref{stadefilab}) gives here (in (iv)) that $i$ takes the values $1,...,t$.
		\item[(v)] Here we just fix a convenient choice among conjugate ones. We require that: $D_\F\subseteq U_N$, all the $V_i$ are generated by $\ossa$-groups (which we recall in Definition \ref{ossagroup}), any two $V_i$ and $V_j$ have no nontrivial entries in common, $Q_{t+1}$ is equal to $\tj(GL_n)$ where $\tj: GL_n\rightarrow GL_N$ is the embedding given by \begin{equation}\label{jlab}\tj(g)=\begin{pmatrix}
		g&&\\&\ddots&\\&&g
		\end{pmatrix}, \end{equation} and $J_{\F|_{V_i}}$ (recalled in Definition \ref{mtrox}) is a  conjugate of a Jordan matrix by an element in $W_N$.
				\end{itemize}
	We denote by $\RRR_{n,N}$ the set of such AFs $\F$. Notations depending on $\F$ without mentioning $\F$ (e.g. $P_i$ and $Q_i$) are assumed to be defined each time with respect to the way $\F$ is chosen in $\RRR_{n,N}$.  \end{defi}
One can directly obtain a more explicit description of $\RRR_{n,N}$; for example, the unipotent radical of  $P_i$ is either abelian or two step nilpotent; we mention more such information in Observation \ref{expr}, but only use it in the proof of Lemma \ref{Blem} (and there it is used freely). 

     Even though each $Q_i$  is reductive if and only if it is a general linear group, we frequently use phrases such as ``$Q_i$ is reductive".

\begin{example}Two examples of $\F\in\RRR_{n,N}$ (with only two ``diagonality levels"\footnote{That is, $x(t)=2$,  where $x(...)$ is defined Definition \ref{papardefi}.}), appear in the two pictures in Remark \ref{prema}. Part of a choice of data describing the choice of $\F$ described in the left (resp. right) picture is: $P_1$ with Levi isomorphic to $GL_{12}\times GL_2$, $t=4$, $P_i=Q_i$ except for $i=2$ (resp. the Levi of $P_1$ is isomorphic to $GL_{14}\times GL_{2}$, $t=5$, and $P_i=Q_i$ except for $i=3$). Of course in the left choice of $\F$ there is a second way to choose the data, that is the Levi of $P_1$ is isomorphic to $GL_6\times GL_8$, $t=4$ (then we always have $P_i=Q_i$). 
	
\end{example}
\begin{lem}[and definition of $\wop$]\label{doubc}
	Let $\F\in\RRR_{n,N}.$ The double cosets of $\Pmir{N}(\kkk)\s GL_N(\kkk)/(\tj(GL_n)D_\F)(\kkk)$ are all represented by elements in $W_N$. We denote by $\wop$ the minimal length element in $W_N$ representing the open double coset.
\end{lem}
\begin{tproof}Let $P$ be the $GL_N$-parabolic subgroup with unipotent radical $D_\F$. For any embedding of the form $h\rightarrow w\begin{pmatrix}
	h&&\\&\ddots\\&&h
	\end{pmatrix}w^{-1}$ for all $h\in GL_m$ for some $m$ dividing $N$, all blocks except one are contained in $\Pmir{N}$. Hence
	 
	  $$\Pmir{N}(\kkk)\s GL_N(\kkk)/P(\kkk)=\Pmir{N}(\kkk)\s GL_N(\kkk)/(\tj(GL_N)D_\F^{\mathrm{root}})(\kkk) $$
	
	 and$$\Pmir{N}(\kkk)\s GL_N(\kkk)/(\tj(GL_N)(V_1...V_{i+1})^{\mathrm{root}}V_{i+2}...V_t)(\kkk)=\Pmir{N}(\kkk)\s GL_N(\kkk)/(\tj(GL_N)(V_1...V_{i})^{\mathrm{root}}V_{i+1}...V_t)(\kkk) $$ where for any $\ossa$-group $L $ we denote by $L^{\mathrm{root}}$ or $(L)^{\mathrm{root}}$ the smallest group containing $L$ which is generated by root groups. From these equalities and the Bruhat decomposition we are done.	\end{tproof}

\begin{defi}[$\F_{\emptyset,H}$]For an algebraic group $H$ we define $\F_{\emptyset,H}$ to be the AF with domain the trivial subgroup of $H$. \end{defi}
	\begin{defi}\label{alF}
		Let $\F$ be an AF and $a:G\rightarrow H$ be an algebraic homomorphism of algebraic groups $G,H,$ such that $D_{\F}$ and $G$ are algebraic subgroups of a (common to both) algebraic group. We call $a(\F)$, if it exists, the AF with domain $a(D_\F\cap G)$ satisfying \begin{equation*}a(\F)(a(u))=\F(u)\qquad\qquad\forall u\in D_\F\cap G.\qedhere \end{equation*}	\end{defi}
\begin{defi}[$\prod,\times$]\label{pstr} In addition to using $\prod$ and $\times$ for direct products of groups we also use them as follows: Given group homomorphisms $j_i:H\rightarrow H_i$ for $1\leq i\leq x$ we denote the homomorphism of $H$ on $\prod_{i}H_i$ mapping each $h\in H $ to $(j_1(h),...,j_x(h))$ by $\prod_{i}j_i$ or by $j_1\times...\times j_x$. For $\Z_1,...,Z_x$ being AFs with domain a unipotent algebraic subgroup of $H$, in case $D_{\Z_1}...D_{\Z_2}$ is a direct product (of the $D_{\Z_i}$), two other (and preferred) names for $\Z_1\circ...\circ\Z_2$ are $\prod_{i}\Z_i$ and $\Z_1\times...\times Z_x$. \end{defi}

 \begin{prop}[and definitions of $\WW_n$, $\tjwa$, $p$, and $Y(\A)$]\label{ftprop}Let $\F\in\RRR_{n,N}$. Let $I,\varphi,\phi,$ and $E$ be as in the abstract, and further assume that $E$ admits an absolutely convergent Eisenstein series expansion $E(g)=\sum_{\gamma\in \Pmir{N}(\kkk)\s GL_N(\kkk)}f(\gamma g)$. Then\begin{equation}I=\label{sumprop}\sum_w\int_{Y_w(\A)}(-\WW_n)(\varphi)(g) (\mathrm{id}\times \tjwa)(\WW_n)\circ(\F_{\emptyset,n}\times p(w\F))(\phi f)(g,wh\tj(g))dhdg\end{equation}where: $\WW_n$ is an AF with domain $U_n$ which is nontrivial on all simple root groups; $p$ is the projection\footnote{That is, $p$ is trivial on the unipotent radical of $\Pmir{N}$ and restricts to the identity function on the Levi.} of $\Pmir{N}$ onto its Levi; $w$ varies over the elements in $W_N$ which are the smallest choice of representatives (one) for each double coset in the lemma above for which $p(w\F)$ is defined (by Definition \ref{alF}) (equivalent to  ``$p(w\F)$ is defined" is that the restriction of $w\F$ on the intersection of $D_{w\F}$ with the unipotent radical of $\Pmir{N}$ is trivial);  $\tjwa$ is the function given by $\tjwa(g)=p(w\tj(g)w^{-1})$ (for all possible $g$); $Y_w(\A):=(Z_n\tj(U_n)D_\F\cap w^{-1}\Pmir{N}w)(\A)\s (\tj(GL_n)D_{\F})(\A)$, and $\mathrm{id}$ is the identity function on $GL_n$.\end{prop}\begin{proof}
	As in familiar special cases of $I$, we unfold it by first using the Eisenstein series expansion of $E$ over $\Pmir{N}$ and then by using the Fourier expansion of $\varphi$ over $U_n(\kkk)\s U_n(\A)$, except that the Fourier expansion is applied to every term obtained by Lemma \ref{doubc} (except the terms for which $p(w\F)$ is not defined).\end{proof}As we mentioned earlier, Definition \ref{defBBB} ($\BBnk[H]$, $\BBnk[H](k)$,...) is needed for Lemma \ref{Blem} below. As for the definitions below preceding the Lemma, (they are of course also needed but) deferring to read them to various extents, is likely efficient; after the proof of Lemma \ref{Blem}, only $A(\F)$ is used among them (in Proposition \ref{prop1}). 
\begin{defi}[$C$-root group, $\car$] 
	
	Let $T_N$ be the maximal standard torus of $GL_N$. For any subgroup $H$ of $GL_N$, let $N_{GL_N}(H)$ be the normalizer of $H$ in $GL_N$ (that is $\{\gamma\in GL_N:\gamma H\gamma^{-1}=H\})$. Let $C$ be a subgroup of $GL_N$, which is generated by $\ossa$-groups and elements in $T_N;$ then for a $\ossa$-group $L$ contained in $C$ for which $T_N\cap N_{GL_N}(C)\subseteq N_{GL_N}(L),$ we say that $L$ is a $C$-root group. Also, for $C$ being generated by $\ossa$-groups, and $L^\prime$ being a subgroup of $GL_N$ generated by root groups, we define $C\car L^\prime$ to be the subgroup of $C$ generated by the $C$-root groups which have a nontrivial entry in common with $L^\prime$.\end{defi}
	
\begin{defi}[$\row{L}$, $\col{L}$, $y_i$, $i_1,i_2,...$, $s(i)$, $x(i)$]\label{papardefi}	For $L$ being a $\ossa$-group let $\row{L}$  (resp. $\col{L}$) be the set of $r$ such that $L$ is nontrivial in the $r$-th row (resp. $r$-th column).
	
 Let $y_i$ be the number of elements (which are disjoint sets) in the set $$ \{\row{L}:L \text{ is a $V_i$-root group and $\F(L)\neq\{0\}$}\}$$ (of course replacing  ``$\row{L}$" with ``$\col{L} $" does not affect $y_i$).
	
	We define  $i_1=1$, and  for $r>1$, we define $i_r$ to be the smallest number for which  $y_{i_r}<y_{i_{r-1}}$. We define $s(i) $ to be the biggest number for which $i_{s(i)}\leq i$.
	
	We define $x(1):=1$, and $$x(i+1):=\left\{\begin{array}{cc}x(i)&\text{if }x^\prime(i+1)=x^\prime(i)\\x(i)+1&\text{if }x^\prime(i+1)>x^\prime(i)\end{array}\right.$$ where $x^\prime(i)$ (which is not mentioned again) is the number of entries on which a $V_i$-root group is nontrivial.\end{defi}

\begin{ob}\label{expr}Let $\F\in\RRR_{n,N}$ and $r$ be chosen so that $i_r$ and $i_{r+1}$ are defined. Then the collection of numbers $y_{i_r},y_{i_r+1},...,y_{i_{r+1}}$ expressed in the same order takes the form \begin{equation}\label{braq}\underbrace{y_{i_r},...y_{i_r}}_{a_1\text{ times}},\underbrace{z_1,...,z_1}_{b_1\text{ times}},\underbrace{y_{i_r},...,y_{i_r}}_{a_2\text{ times}},\underbrace{z_{2},...,z_2}_{b_2\text{ times}},....,\underbrace{y_{i_r},...,y_{i_r}}_{a_k\text{ times}},\underbrace{z_k,...,z_k}_{b_k \text{ times}},\underbrace{y_{i_r},...,y_{i_r}}_{a_{k+1}\text{ times}},y_{i_{r+1}}\end{equation}where: in the cases not otherwise specified, any of the numbers appearing in (\ref{braq}) (e.g. $k$) can even be zero; $z_j>y_{i_r}$; if $i_r< i< i_{r+1}$, we have $$x(i+1)> x(i)\iff( y_i=z_j \text{ for some $j$ and }y_{i+1}\neq y_i);$$ if $Q_{i_r}$ is not reductive we have $a_1=1$, $b_1=0$, and $x(i_r+1)>x(i_r) $; if $b_j=0$ then $j=1$; if $b_j>0$ we have $$z_j=\left\{\begin{array}{cc}a_{j+1}y_{i_r}+b_{j+1}z_{j+1}&\text{if }1\leq j<k\\a_{k+1}y_{i_r}+y_{i_{r+1}}&\text{if }j=k\end{array}\right..$$  
\end{ob}Recall with the superscript t we denote transpose. \begin{defi}[$A(\F)$]\label{afdef}We define $A(\F) $ to be the set of $w$ as in (\ref{sumprop}) \textbf{except} the ones for which $\alpha$ or  $\beta$ or $\gamma$ hold.
\begin{itemize}
	\item [$\alpha$]  There is an $i$ such that; $V_i\car w^{-1}{\Umir{N}}^tw$ is nontrivial, $V_{i}$ is nontrivial in more rows than columns; $i_{s(i)}<i<i_{s(i)+1}$ and $\alpha$.1 or $\alpha$.2 below holds.\begin{itemize}\item[$\alpha$.1] $x(i)=x(i+1)$. $\F(V_{i+1}\car w^{-1}\Umir{N} w)\neq\{0\}$ (note that for $i+1<i_{s(i)+1}$ this is implied from the part of the sentence preceding the last ``;"). There is a number $i^\prime$  such that $i_{s(i)}\leq i^\prime<i$, $y_{i^\prime}=y_{i_{s(i)}}$ and $x(i^\prime)<x(i)$; choose the biggest such $i^\prime$; then for the $V_i\car w^{-1}{\Umir{N}}^tw$-root group, say $L$, on which $\F$ is nontrivial we have $$\col{L}\cap\left(\bigcup_{L^\prime}\row{L^\prime}\right)=\emptyset$$	where $L^\prime$ varies over the $V_{i^\prime}$-root groups satisfying  $\F(L^\prime)\neq\{0\}$.\item[$\alpha$.2] $Q_{i_{s(i)}}$ is reductive and $x(i_{s(i)}+1)=x(i_{s(i)+1})$.\end{itemize}
	\item [$\beta$] There is an $s$ such that: $V_{i_s}$ intersects more columns than rows; $V_{i_s}\car w^{-1}{\Umir{N}}^t w$ is trivial; and $\beta$.1 or $\beta$.2 below holds.\begin{itemize}
		\item[$\beta.$1] $x(i_{s+1})>x(i_s)+1$; for every $i$ for which $i_{s}<i<i_{s+1}$ and $y_{i-1}<y_i$, the group $V_i\car w^{-1}{\Umir{N}}^t w$ is trivial;		
		\item[$\beta$.2] $Q_{i_s} $ is reductive.
	\end{itemize}
\item[$\gamma$] There is an $s$ such that $V_i\car w^{-1}{\Umir{N}}^t w$ is trivial for all $i_s\leq i\leq i_{s+1}$, and $Q_{i_s}$ is reductive.
\end{itemize}\end{defi}
\begin{remark}Of course $\wop\in A(\F)$. For an $\F\in\RRR_{n,N}$ for which $D_\F$ is not generated by root groups, we have $A(\F)=\{\wop\}$ if and only if: for each $r$, at least one among $Q_{i_{r}}$, and $Q_{i_{r+1}}$ is reductive; and if $Q_{i_{r+1}}$ is not reductive then $V_{i_{r}}$ is nontrivial in more rows than in columns and $y_i=y_{i_{r}}$ for all $i_{r}\leq i< i_{r+1}$.
	
\end{remark}
\begin{lem}\label{Blem}Let $\F\in\RRR_{n,N}$. Consider any $w$ as in (\ref{sumprop}). Then:\begin{itemize}
		\item[(i)] $\F\in\BBnk[N]$;
		\item[(ii)] $p(w\F)\in\BBnk[\Mmir{N}](\Di{D_{p(w\F)}}-\Di{D_{p(\wop\F)}})$;
					\item[(iii)] For $w\not\in A(\F)$ we have $p(w\F)\in\BBnk[\Mmir{N}](\Di{D_{p(w\F)}}-\Di{D_{p(\wop\F)}}-1); $ 
		\item[(iv)] There is an $(\F_{\emptyset,n}\times p(w\F)\rightarrow(\mathrm{id}\times \tjwa)(W_n)\circ (\F_{\emptyset,n}\times p(w\F)),\BBnk[GL_n\times \Mmir{N}]) $-path (note for example this together with (ii) implies that $(\mathrm{id}\times \tjw)(W_n)\circ (\F_{\emptyset,n}\times p(\wop\F))\in\BBnk[GL_n\times \Mmir{N}]$). 
	\end{itemize} \end{lem}
\begin{tproof}Consider the $(\F_{\emptyset,N}\rightarrow \F)$-path $\Xi$ with labels of vertices given at increasing depth: $\F_{\emptyset,N}$, $ \F|_{V_1}$, $\F|_{V_1V_2}$, ..., $\F$. For each $1\leq i\leq t$ that is possible, let $H_i^1$ (resp. $H_i^2$) be an algebraic subgroup of $P_i$ which: \begin{itemize}
		\item acts by conjugation freely and transitively on an open subset of the set of terms of the $\e$-step of $\quasi{\Xi}$ with input AF $\F|_{V_1...V_{i-1}}$, and this open subset contains $\F|_{V_1...V_i}$.
		\item every matrix in it is trivial on all nondiagonal entries $(r,c)$ with $r,c$ belonging in \begin{equation}\label{bset}\{x:\F|_{V_i}\text{ is nontrivial on a $V_i$-root group, such that this group is nontrivial on the $x$-th row (resp. column)}\}.\end{equation}
	\end{itemize}
We see that for each $1\leq i\leq t$ at least one between $H_i^1$ and $H_i^2$ is defined, and hence $\Xi$ is a $\BBnk[N]$-path, and hence we obtain (i).	
	
	Unless more restrictively specified, let $w$ be any element as in (\ref{sumprop}).

	  Given an edge of an $\AAnk$-tree, we identify it with the $\AAnk$-path with only two vertices  for which this is the (unique) edge it has. Also if such a path is a $\BBnk[H](k)$-path (for some choice of $H$ and $k$), we also call it a $\BBnk[H](k)$-edge.
	  
	  Let $X$ be a $\ossa$-group and $X^\prime$ be a $V_i$-root group. Assume that there are numbers $l_1,...,l_k$ such that: $X$ is nontrivial on the $l_1$-th row (resp. column); for each $1<r\leq k$,  there is a $j$ with $y_{j}=y_i$, and a $V_j$-root group, say $X_r$, which is nontrivial on the $(l_{r},l_{r-1})$ entry (resp. $(l_{r-1},l_{r})$ entry) and such that $\F(L_r)\neq\{0\}$; and $X_k=X^\prime$. Then we say that $X$ is linked to $X^\prime$.
	
	 Let $\Xi_w$ be the $\AAnk$-path obtained by replacing in $\Xi$ each label $\F|_{V_1...V_i}$ with the label $p(w(\F|_{V_1...V_i}))$. Consider the $i$-th edge of $\Xi_w$; then notice that one among $wH_i^1w^{-1}$ and $wH_i^2w^{-1}$, contains a subgroup the action of which implies that this edge is a $\BBnk[\Mmir{N}](k_i)$-edge where\begin{itemize}
		\item[(a)] $k_i=0$, if one among (a.1), (a.2), and (a.3) below hold (of course (a.1.1) (resp. (a.1.2), (a.3.1)) is considered part of (a.1) (resp. (a.1), (a.3)):\begin{itemize}
			\item[(a.1)] $V_i\car w^{-1}\Umir{N} w$ and $V_i\car w^{-1}{\Umir{N}}^t w$ are both trivial; and (a.1.1) or (a.1.2) holds.
			
			\vspace{0.4mm}
			\noindent{(a.1.1)} (resp. (a.1.2))\hspace{1.5mm} In case $i=i_{s(i)}$, $Q_{i}$ is not reductive, and $V_i$ intersects nontrivially more rows (resp. columns) than columns (resp. rows), then $\F$ is trivial on all $V_i$-root groups (resp. nontrivial on a $V_i$-root group) linked to a: $w^{-1}V_jw$-root group contained in $w^{-1}{\Umir{N}}^t w$ or in $w^{-1}\Umir{N}w$ for a $j$ satisfying $i_{s(i)-1}\leq j<i_{s(i)}$. 		
			\item[(a.2)] $V_i\car w^{-1}{\Umir{N}}^t w$ is nontrivial, and either $V_{i}$  intersects nontrivially at least as many columns as rows or $\F(V_i\car w^{-1}{\Umir{N}}^t w)=\{0\}$;
			\item [(a.3)] $\F(V_i\car w^{-1}\Umir{N} w)\neq\{0\}$; and (a.3.1): in case $V_{i_{s(i)}}$ nontrivially intersects more columns than rows we also require that a $V_{i_{s(i)}}$-root group on which $\F$ is nontrivial is linked to a $w^{-1}V_iw$-root group contained in $w^{-1}\Umir{N} w$.
		\end{itemize} 
		\item[(b)] $k_i=y_i$, if: $V_i\car w^{-1}\Umir{N} w$ is nontrivial and $\F(V_i\car w^{-1}\Umir{N} w)=\{0\}$; or the negation of (a.1.2) holds. 
		\item[(c)] $k_i=y_i-y_{i_{s(i)}}$, if: $\F(V_i\car w^{-1}\Umir{N} w)\neq\{0\}$ and the negation of (a.3.1) holds.
		\item[(d)] $k_i=\sum_{i<j<i_{s(i)+1}\text{ and }x(i)\leq x(j)}y_j$ if: either (all three) (d.1) (d.2) and (d.3) hold, or the negation of (a.1.1) holds.\\ $k_i=\sum_{i<j\text{ and }x(i)=x(j)} y_j$  if: (d.1) (d.2) and the negation of (d.3) hold. \begin{itemize}\item[(d,1)]$V_i$ nontrivially intersects more rows than columns.\item[(d,2)]  $\F(V_i\car w^{-1}{\Umir{N}}^tw)\neq\{0\}$.
		\item[(d.3)] A $V_{i_{s(i)}}$-root group on which $\F$ is nontrivial is linked to a $V_i\car w^{-1}{\Umir{N}}^tw$-root group.\end{itemize}
	\end{itemize}Hence (since $\Xi_w$ is a $\BBnk[\Mmir{N}](\sum_i k_i)$-path) we obtain (ii), and the restriction of (iii) to the $w$ with the added condition: $\alpha$ does not hold (``$\alpha$" appears in the definition of $A(\F)$). Hence consider a choice of $w$ for which $\alpha$ holds. Then we may have $\sum_i k_i=\Di{D_{p(w\F)}}-\Di{D_{p(\wop\F)}}$, and hence we will describe a modification $\Xi_w^\prime$ of $\Xi_w$ which will turn out to be an $(\F_{\emptyset,\Mmir{N}}\rightarrow p(w\F),\BBnk[\Mmir{N}](\Di{D_{p(w\F)}}-\Di{D_{p(\wop\F)}}-1))$-path. Let $L_i$ be the  $V_i\car w^{-1}{\Umir{N}}^tw$-root group such that $\F(L_i)\neq\{0\}$. We define $\Xi_w^\prime$ to be obtained from $\Xi_w$ so that:\begin{itemize}\item If $\alpha$.1 holds, we replace the vertex with label $p(w(\F|_{V_1...V_i}))$ with the $\e$-edge with input and output AFs respectively being   $p(w(\F|_{V_1...V_{i-1}V^1_i}))$ and $p(w(\F|_{V_1...V_{i-1}V_1^2V_{i+1}})),$ where $V_i^1$ (resp. $V_i^2$) is  generated by all the $V_i$-root groups except the ones, say $L$, for which $\col{L}=\col{L_i}$ (resp. (again) $\col{L}=\col{L_i}$ and  $$\row{L}\in\{\row{L^\prime}:L^\prime\text{ is a $V_i$-root group with }\F(L^\prime)\neq\{0\}\}).$$	
	
\item Consider now the case that $\alpha$.1 does not hold. Hence $\F(V_{i+1}\car w^{-1}\Umir{N} w)=\{0\}$, $i=i_{s(i)+1}-1$, $Q_{i_{s(i)}}$ is reductive, $x(i_{s(i)})=x(i_{s(i)+1})$ and $y_j$ is the same for all $i_{s(i)}\leq j\leq i$. Hence there is an alternative way we can choose $V_{j}$ for $i_{s(i)}\leq j\leq i_{s(i)+1}$ (and still obtain the same $\F$), which we denote by $V_j^\prime$. Then replace in $\Xi_w$ the $\AAnk$-subpath with labels $p(w(\F|_{V_1...V_{i_{s(i)}+1}})$,...,$p(w(\F|_{V_1...V_{i_{s(i)+1}}}))$, with the $\AAnk$-path with labels \begin{equation}\label{verlab}p(w(\F|_{V_1...V_{i_{s(i)}}V^\prime_{i_{s(i)}+1}}),...,p(w(\F|_{V_1...V_{i_{s(i)}}V^\prime_{i_{s(i)}+1}...V^\prime_{i_{s(i)+1}}})).\end{equation} \end{itemize}
We see that:\begin{itemize}\item If $\alpha$.1 holds, then in the three (successive) edges in $\Xi_w^\prime$ not encountered in $\Xi_w$ we have:\begin{itemize}\item the first (that is the one with input AF being $p(w\F|_{V_1...V_{i-1}})$) and last one is a $\BBnk[\Mmir{N}]$-edge;\item the middle one is a $\BBnk[\Mmir{N}]((\sum_{j} y_j)-1)$-edge, where in case $y_i=y_{i_{s(i)}}$ (resp. $y_i>y_{i_{s(i)}}$) $j$ varies over the values satisfying $i<j\leq i_{s(i)+1}$ (resp. $i<j$ and $y_i=y_j$).\item In case $\F$ is  trivial in all $V_{i_{s(i)}}$-root groups (resp. nontrivial in a $V_{i_{s(i)}}$-root group) linked to $L_i$, the last edge is a $\BBnk[\Mmir{N}]$-edge (resp. $\BBnk[\Mmir{N}](y_{i}-y_{i_{s(i)}})$-edge).\end{itemize}\item If $\alpha$.1 does not hold, then the path with vertices as in (\ref{verlab}) is a $\BBnk[\Mmir{N}](y_{i_{s(i)+1}})$-path (because its first edge is a $\BBnk[\Mmir{N}](y_{i_{s(i)+1}})$-edge and the next ones are $\BBnk[\Mmir{N}]$-edges).\end{itemize}
Hence $\Xi_w^\prime$ is as claimed (in the first sentence it was mentioned) and (iii) is (fully) obtained.

Since $\Stab{GL_n\times \Mmir{N}}{p(w\F)}$ contains an appropriate copy of the mirabolic subgroup of $GL_n$, by applying $\mathrm{id}\times \tjwa$ and then $\circ(\F_{\emptyset,n}\times p(w\F))$ to an appropriate $(\F_{\emptyset,n}\rightarrow\WW_n,\BBnk[n])$-path, we obtain (iv).\end{tproof}
\begin{remark}In the proof of Proposition \ref{prop1} we obtain bounds on the dimensions of orbits by using Main corollary \ref{maincor} and the Lemma above. More precicely the uses of the Lemma are $\ast$ together with (iv) for $\ast=$(i), (ii), (iii). Even though not used, it may worth mentioning that the bounds obtained in the same way by only using (ii) from the Lemma above---that is, $\frac{\Di{a}}{2}\geq \Di{D_{p(\wop\F)}}$ for every $a\in\Om(p(w\F))$--- can also be obtained as follows.
	
	Among the $\ossa$-groups which are contained in $\Stab{GL_N}{\F}$ and are nontrivial in the entry $(n,i)$, let $X_i$ be the one which is nontrivial in as few entries as possible. Let $\F^\prime$ be an AF which: has domain $D_{\F^\prime}=(\prod_{1\leq i\leq n-1} X_i)D_\F$, is nontrivial in  $\prod_{1\leq i\leq n-1} X_i$, and $\F^\prime|_{D_\F}=\F$. By (i) in the Lemma above we easily obtain $\F^\prime\in\BBnk[N]$. For a $J\in\X{p(w\F)}$ choose a $J^\prime\in\X{w\F^\prime}\cap\Lie{\Pmir{N}}^t$ which projects (through the differential of $p$) to $J$. Hence (by 7.1.1 in \cite{Col}) $\Di{\Om(J)}\geq\Di{\Om(J^\prime)}-N+1$, and by also using Main corollary \ref{maincor} for $[\F\tlarrow\F^\prime]$ we are done.\end{remark}\begin{remark}\label{prema}
	The dimension bound obtained from (i) in the Lemma above and Main corollary \ref{maincor} is optimal. That is there is an $a\in\Om(\F)$ (resp. $a\in\Om(p(\wop\F))$) satisfying $\frac{\Di{a}}{2}=\Di{D_\F}$ (resp. $\frac{\Di{a}}{2}=\Di{D_{p(\wop\F)}}$); we obtain such an orbit as in the proof of Lemma 3 in \cite{Tsiokos3}. 
	
	In the picture on the left (resp. right) below, we describe a choice of $\F$ for which the dimension bound obtained as in the previous paragraph  by replacing (i) with (ii) (resp. (ii) and (iv)) is optimal. The picture rules are the same as in \cite{Tsiokos2} (see Definition 8.2.3 and the paragraph slightly below this definition finishing with ``$\triangle$fixing a picture").  
	
	More precisely, choose $a=[3,3,4,3]$ (resp. $a=[4,2,5,4]$), $w$ so that for $V$ being a positive root group in the forth row, the root group $wVw^{-1}$ is negative; then one can check that $a\in \Om(p(w\F))$ (resp. $a\in\Om(p(w(\tj(\WW_2)\circ\F)))$) and $\frac{\Di{a}}{2}=\Di{D_{p(\wop\F)}}$ (resp. $\frac{\Di{a}}{2}=\Di{D_{p(\wop(\tj(\WW_2)\circ\F))}}$). The order of the numbers inside the brackets---which does not affect the meaning of $a$--- is chosen in the way one ``encounters them" by following an appropriate initial $\AAnk$-subpath of $\sXi(\F)$ with its output AF in $\Prink[13][a]$ (resp. $\Prink[15][a]$). 
	
	\renewcommand{\di}{14}
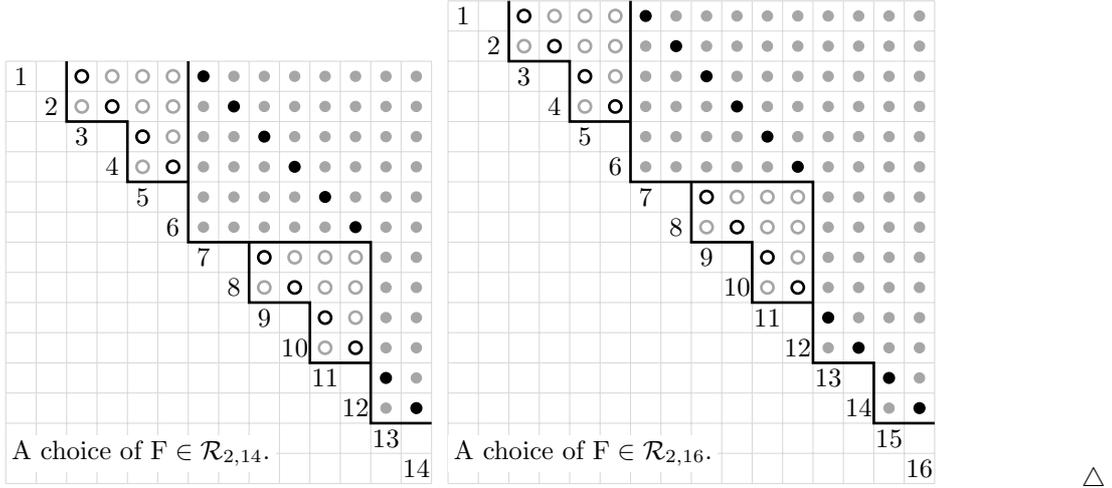
	\begin{ti}[0.4]
		\grid{14}
		\f{1}{3/d,4/c,5/c,6/c,7/b,8/a,9/a,10/a,11/a,12/a,13/a,14/a}
	\f{2}{3/c,4/d,5/c,6/c,7/a,8/b,9/a,10/a,11/a,12/a,13/a,14/a}
	\f{3}{5/d,6/c,7/a,8/a,9/b,10/a,11/a,12/a,13/a,14/a}
	\f{4}{5/c,6/d,7/a,8/a,9/a,10/b,11/a,12/a,13/a,14/a}
	\f{5}{7/a,8/a,9/a,10/a,11/b,12/a,13/a,14/a}
	\f{6}{7/a,8/a,9/a,10/a,11/a,12/b,13/a,14/a}
	\f{7}{9/d,10/c,11/c,12/c,13/a,14/a}
	\f{8}{9/c,10/d,11/c,12/c,13/a,14/a}
	\f{9}{11/d,12/c,13/a,14/a}
	\f{10}{11/c,12/d,13/a,14/a}
	\f{11}{13/b,14/a}
	\f{12}{13/a,14/b}
	\f{13}{}
	\f{14}{}
	\draw[line width=1 pt]\mm{0}{6}--\mm{6}{6}--\mm{6}{12}--\mm{12}{12}--\mm{12}{14};	
	\draw[line width=1 pt]\mm{0}{2}--\mm{2}{2}--\mm{2}{4}--\mm{4}{4}--\mm{4}{6};
	\draw[line width=1 pt]\mm{6}{8}--\mm{8}{8}--\mm{8}{10}--\mm{10}{10}--\mm{10}{12};
	\fir{\text{A choice of }\F\in\RRR_{2,14}.}
	\end{ti}\renewcommand{\di}{16}
\begin{ti}[0.4]
	\grid{16}
	\f{1}{3/d,4/c,5/c,6/c,7/b,8/a,9/a,10/a,11/a,12/a,13/a,14/a,15/a,16/a}
	\f{2}{3/c,4/d,5/c,6/c,7/a,8/b,9/a,10/a,11/a,12/a,13/a,14/a,15/a,16/a}
	\f{3}{5/d,6/c,7/a,8/a,9/b,10/a,11/a,12/a,13/a,14/a,15/a,16/a}
	\f{4}{5/c,6/d,7/a,8/a,9/a,10/b,11/a,12/a,13/a,14/a,15/a,16/a}
	\f{5}{7/a,8/a,9/a,10/a,11/b,12/a,13/a,14/a,15/a,16/a}
	\f{6}{7/a,8/a,9/a,10/a,11/a,12/b,13/a,14/a,15/a,16/a}
	\f{7}{9/d,10/c,11/c,12/c,13/a,14/a,15/a,16/a}
	\f{8}{9/c,10/d,11/c,12/c,13/a,14/a,15/a,16/a}
	\f{9}{11/d,12/c,13/a,14/a,15/a,16/a}
	\f{10}{11/c,12/d,13/a,14/a,15/a,16/a}
	\f{11}{13/b,14/a,15/a,16/a}
	\f{12}{13/a,14/b,15/a,16/a}
	\f{13}{15/b,16/a}
	\f{14}{,15/a,16/b}
	\f{15}{}
	\f{16}{}
	\draw[line width=1 pt]\mm{0}{6}--\mm{6}{6}--\mm{6}{12}--\mm{12}{12}--\mm{12}{14}--\mm{14}{14}--\mm{14}{16};
	\draw[line width=1 pt]\mm{0}{2}--\mm{2}{2}--\mm{2}{4}--\mm{4}{4}--\mm{4}{6};
	\draw[line width=1 pt]\mm{6}{8}--\mm{8}{8}--\mm{8}{10}--\mm{10}{10}--\mm{10}{12};
	\fir{\text{A choice of }\F\in\RRR_{2,16}.}
\end{ti}
\end{remark}
In the rest of the present paper we concentrate on the $I$ ``belonging" in a set $\SSS$ defined below, for which the choices of $\phi$ are products of certain Fourier coefficients of automorphic forms. The scare quotes on the word belonging are because $\SSS$ is a set of integral expressions (precisely described as tuples and in a way specific to the needs of the present paper).
\begin{defi}[$\SSS$]\label{sssdefi} We denote by $\SSS$ the set of tuples of the form $$\eI:=(\F_1,...,\F_{k-1},\F,\pi_{\mathrm{cusp}},\pi_1,...,\pi_{k-1},\pi) $$ where: $\F\in\RRR_{n,N}$, $\pi_{\mathrm{cusp}}$ is a $GL_n(\A)$-cuspidal automorphic representation, $\pi$ is a $GL_N(\A)$-automorphic representation,  $k$ is a positive integer; and for $1\leq i\leq k-1$ we define positive integers $N_i$, AFs $\F_i\in\BBnk[N_i]$, embeddings $\tj_i: GL_n\rightarrow GL_{N_i}$ for which $\tj_i(GL_n)\in\Stab{GL_{N_i}}{\F_i}$,  and $GL_{N_i}(\A)$-automorphic representations $\pi_i$ such that the vector space $\{\text{restriction of }\F(\varphi_i)\text{ on }\tj_i(GL_n(\A)):\varphi_i\in\pi_i\}$ is neither trivial nor one dimensional. 
	
	For the tuple $\eI$ in $\SSS$ we obtain the functional $I:\pi_{\mathrm{cusp}}\times \prod_{i\leq i\leq k-1}\pi_i\times\pi\rightarrow\mathbb{C}$, which maps each tuple $(\varphi,\varphi_1,...,\varphi_{k-1},E)$ in its domain to the integral \begin{equation}\label{ari}I(\varphi,\varphi_1,...,\varphi_{k-1},E)=\int_{GL_n(\kkk)Z_n(\A)\s GL_n(\A)}\varphi(g)\left(\prod_{1\leq i\leq k-1}\F_i(\varphi_i)(\tj_i(g))\right)\F(E)(\tj(g))dg. \end{equation}  
	
	 Whenever we have a text of the form ``$I=\ast$" for $\ast$ being any text, we mean ``$I(\varphi,\varphi_1,...,\varphi_{k-1},E)=\ast$ for all $\varphi,\varphi_1,...,\varphi_{k-1},E$ in the domain of $I$".
 
To be clear, whenever an $\eI\in\SSS$ appears, all notations within the present definition (e.g. $I$) are adopted for this choice of $\eI$. \end{defi}
\begin{defi}[$\UUU{H}$, $\Om(X)$, blocks]\label{lakm}
	Let $H=\prod_{1\leq i\leq k}GL_{n_i}$ (for some positive integers $k,n_1,...,n_k$). Then we refer to each $GL_{n_i}$ as a block. We denote by $\UUU{H}$ the set of nilpotent orbits of the action by conjugation of $H$ on $\Lie{H}.$ We mostly write $\UUU{n}$ instead of $\UUU{GL_n}$ and we identify its elements with positive integer partitions $n=n_1+...$ which we denote by $[n_1,...]$.  For $X\in\Lie{H}$ we denote by $\Om(X)$ the orbit in $\UUU{H}$ that contains $X$.  
	
	As usual the order we use for $\UUU{H}$ is: given two different $a,b\in\UUU{H}$ with $a=[a_{1,1},a_{1,2}...]\times[a_{2,1},...]\times...$ and $b=[b_{1,1},b_{1,2}...]\times[b_{2,1},...]\times...$, we say  $a$ is bigger from $b$ if and only if $\sum_{1\leq j\leq k}a_{i,j}\geq\sum_{1\leq j\leq k}b_{i,j}$ (for every possible choice of $i,k$).
	
	For an $a\in\UUU{H}$ consider the (again) form $a=a_1\times...\times a_k,$ for $a_i\in\UUU{n_i}$; then if $a_i$ is trivial (resp. nontrivial) we say that $a$ is trivial (resp. nontrivial) on the $i$-th block. Usually we prefer to call leftmost (resp. rightmost) block the first (resp. last) block. \end{defi}
\begin{defi}[$\Di{\pi}$]Let $\pi$ be a $GL_N(\A)$-automorphic representation. Let $\Om(\pi)$ be the set consisting of the maximal orbits of $$\{\Om(J_\F):\F\in\Pri[,N]\text{ and }\F(\pi)\neq 0\} $$ where the definition of the matrix $J_\F$ (the set $\Prink[N]$) is recalled (and extended) in Definition \ref{mtrox} (resp. \ref{pridefi}). The equivalence of this definition of $\Om(\pi)$ with the one found for example in \cite{Ginzsmall} follows from the proof of Corollary 9.3.5 in \cite{Tsiokos2}.
	
	To my knowledge, the proof that $\Om(\pi)$ is a singleton has not appeared for all choices of $\pi$; hence we fix throughout the paper a choice, say $a_{\pi}$, in $\Om(\pi)$, and we define $\Di{\pi}:=\frac{\Di{a_{\pi}}}{2}$.\end{defi}
\begin{defi}[Dimension equation of D. Ginzburg, $\sd(\eI)$]\label{dequ}D. Ginzburg has formulated a dimension equation which is satisfied by many familiar Rankin-Selberg integral expressions, and is significantly related to them being factorizable. Without making any change to the equation, we extend the set of integral expressions for which D. Ginzburg formulated\footnote{By ``formulate" I only mean to precisely state the equation (it does not have to be satisfied). The integral expressions for which D. Ginzburg formulated the equation include the ones in \cite{GinzAn} in (2).} it, so that it contains the elements in $\SSS$ when they are viewed as integral expressions in the way ``suggested" in Definition \ref{sssdefi}. Main corollary \ref{maincor} justifies to some extent not making any change. Let $\eI\in\SSS$. We define $\sd(\eI)$ to be the difference of the two sides of this equation, and hence D. Ginzburgs equation reads: $\sd(\eI)=0$. Since cuspidal representations are generic we have: 
\begin{equation}\label{gdeqd}\sd(\eI)=\left(\sum_i \Di{\pi_i}\right)+\Di{\pi}-\frac{n(n+1)}{2}+1-\left(\sum_i\Di{D_{\F_i}}\right)-\Di{D_\F}.\end{equation}  \end{defi}

\begin{prop}\label{prop1}
	Consider an $\eI\in\SSS$ such that the elements in $\pi$ admit an absolutely convergent Eisenstein series expansion over $\Pmir{N}$. Let $A_0$ be the set consisting of the $w$ as in \ref{sumprop} for which $\Om((\prod\tj_i\times \tjwa)(\WW_n)\circ( \prod_i\F_i\times p(w\F)))$ contains an orbit with dimension $2\Di{D_{p(\wop\F)}}+2\sum_i(\Di{D_{\F_i}})+n(n-1)$ (of course by Lemma \ref{Blem} and Main corollary \ref{maincor}, the dimension couldn't be any smaller, and  $A_0\subseteq A(\F)$). \begin{itemize}\item[i] If $\sd(\eI)=0$ and $k=1$, then $$I=\sum_{w\in A_0}\int_{Y_{w}(\A)}(-\WW_n)(\varphi)(g)  p(w(\tj(\WW_n)\circ\F))( f)(w h\tj(g))dhdg. $$ \item[ii] If $\sd(\eI)=0$, $k=2$, and $(\tj_1(\WW_n)\circ\F_1)(\pi_1)=0$, then $$I=\sum_{w\in A_0}\int_{Y_{w}(\A)}(-\WW_n)(\varphi)(g)(\tj_1(\WW_{n,1}^{\deg})\circ\F_1)(\varphi_1)(g) p(w(\tj(\WW_n^{\deg})\circ\F))(f)(w h\tj(g))dhdg,$$where $\WW_{n,1}^{\deg}$ (resp. $\WW_{n}^{\deg} $) is an AF with domain $U_n$ which is nontrivial exactly on the root groups $U_{n,(1,2)},U_{n,(3,4)},...$ (resp. $U_{n,(2,3)}, U_{n,(4,5)},...$). \item[iii] If $\sd(\eI)=0$, $k=2$, and $(\tj_1(\WW_n)\circ\F_1)(\pi_1)\neq 0$, then$$I=\sum_{w\in A_0}\int_{Y_{w}(\A)}(-\WW_n)(\varphi)(g)(\tj_1(\WW_n)\circ\F_1)(\varphi_1)(g) p(w(\F))(f)(w h\tj(g))dhdg.$$\item[iv] If $\sd(\eI)<0$ or $k>2$, then $I=0$.\end{itemize}
\end{prop}\begin{remark}
Some readers may prefer to only consider this proposition in the special case $N_1=...=N_{k-1}=n$ (and hence $\F_1=...=\F_{k-1}=\F_{\emptyset,n}$) due to the following expectation:\begin{expectation}
	Consider a $\kkk$-AF $\F\in\BBnk[N]$, and a $GL_N(\A)$-automorphic representation $\pi$. We assume that a $\kkk$-copy of $GL_n$ in $GL_N$ is contained in $\Stab{GL_N}{\F}$, and that $0\leq \Di{\F}-\Di{\pi}\leq \frac{n(n-1)}{2}$. We identify $GL_n$ with this copy. Then there is a $GL_n(\A)$-automorphic representation  $\pi^\prime$ such that $$\Di{\pi^\prime}=\Di{\pi}-\Di{D_\F} $$ and for $\F^\prime\in\Pri[,n][\Om(\pi^\prime)]$ we have $\{(\F^\prime\circ\F)(\varphi):\varphi\in\pi\}=\{\F^\prime(\varphi^\prime):\varphi^\prime\in\pi^\prime\}. $\end{expectation}
\end{remark}\begin{proof}[\textbf{Proof of Proposition \ref{prop1}}]To be clear, all nilpotent orbits that are mentioned in the present proof are in $\UUU{GL_{N_1}\times...GL_{N_{k-1}}\times\Mmir{N}}$. By Lemma \ref{Blem} and  Main corollary \ref{maincor} we have:\begin{itemize}
	\item[I] all the terms of $I$ in (\ref{sumprop}) except possibly the ones corresponding to $w\in A_0$ vanish;
	\item[II] if $\sd(\eI)<0$, the terms corresponding to $w\in A_0$  also vanish;
	\item[III] For the special case of the proposition in which $\sd(\eI)=0$ and $k>2$, we are left with proving the following claim:
	
	\vsp\textbf{Claim. }\textit{For $w\in A_0, $ fix an orbit in \begin{equation}\label{ommg}\Omsd{\left(\left(\prod_i\tj_i\right)\times\tjwa\right)(\WW_n)\circ\left(\left(\prod_i\F_i\right)\times p(w\F)\right)}.\end{equation} Then there are at most two blocks on which this orbit is nontrivial; and if they are two, one of them is the rightmost one.
}	\item[IV] if $\sd(\eI)=0$ and $k=2$, it is sufficient for $w\in A_0$ to find an $(j_1\times\tjwa)(\WW_n)\circ(\F_1\times p(w\F))\rightarrow(\tj_1(\WW_{n,1}^{\deg})\circ\F_1)\times p(w(\tj(\WW_{n}^{\deg})\circ\F)))$-path---it is trivial to find one (only use $\e$-steps)---and prove that \begin{equation}\label{ommomm}\Omsd{(\tj_1\times\tjwa)(\WW_n)\circ(\F_1\times p(w\F))}\cap\Om\subseteq\Om((\tj_1(\WW_{n,1}^{\deg})\circ\F_1)\times  p(w(\tj(\WW_{n}^{\deg})\circ\F)))\end{equation} where $\Om$  consists of the orbits in $\UUU{GL_{N_1}\times\Mmir{N}}$ which are nontrivial in both blocks.
\end{itemize}

We are left with proving: the Claim, and (\ref{ommomm}) (with the conditions in IV). We make plenty of use of ``resp." and any two uses in different sentences are independent. In case $\F$ is nontrivial (resp. trivial) let $H:=\prod_{1\leq i\leq k-1}GL_{n}\times \Pmir{n}$ (resp. $H:=\prod_{1\leq i\leq k-1}GL_{n}\times GL_{n-1}$). Let $\tjh$ be the embedding of $U_n$ in $H$ given by $\tjh(u)=(u,...,u,u_k)$ where in case $\F$ is nontrivial we have $u_k:=u$ and in case $\F$ is trivial we define every entry, say $(i,j)$, of $u_k$ to be the $(i+1,j+i)$ entry of $u$.  In case $\F$ is nontrivial or $k>2$ (resp. $\F$ is trivial and $k=2$), we define:\begin{itemize}
	\item $R_1,R_2,...$ to respectively be $U_{n,(1,n)}$, $U_{n,(1,n-1)}$,...,$U_{n,(1,2)}$ (resp. $U_{n,(2,n)}$, $U_{n,(2,n-1)}$,...,$U_{n,(2,3)}$);
	\item $\tjhp$ and $\tjhpp$ to be the embeddings of $U_n$ in $H$ such that they are both equal to $\tjh$ when we restrict to any root group except for $R_1,R_2,...$, and (for each i):  $\tjhp(R_i)$  differs from $\tjh(R_i)$ only on the leftmost block of $H$ and in the leftmost block is trivial;  $\tjhpp(R_i)$ is trivial on all the blocks except the leftmost block, and in the leftmost block is equal to $R_i$.
\end{itemize}  We define by $\tjhp(\WW_n)^+$ (resp. $\tjhpp(\WW_n^+)$) to be the trivial extension of $\tjhp(\WW_n)$ (resp. $\tjhpp(\WW_n)$) with domain $\tjhp(U_n)\tjhpp(U_n)$. Consider the $(\tjh(\WW_n),\e)$-quasipath (resp. $(\tjhp(\WW_n),\e)$-quasipath; $(\tjhpp(\WW_n),\e)$-quasipath) over $\tjhpp(R_1)$, $\tjhpp(R_2),...$  (resp. $\tjhpp(R_1),\tjhpp(R_2),...$ ; $\tjhp(R_1),\tjhp(R_2),...$). Notice that in any $\e$-step in any of these three quasipaths, all the terms except the ones that are constant with respect to a $\ossa$-group are conjugate by an element in $H$. Hence we obtain:  \begin{equation}\label{treesf}\tjh(\WW_n)\riff{H}\{\tjhp(\WW_n)^+,\tjhpp(\WW_n)^+\}\qquad\tjhp(\WW_n)\riff{H}\tjhp(\WW_n)^+\qquad\tjhpp(\WW_n)\riff{H}\tjhpp(\WW_n)^+. \end{equation}Let $\XXX^\prime$ be the set consisting of the AFs in $\AAnk[H]$  which are nontrivial on the leftmost block. In case $k=2$ (resp.  $k>2$) let $\XXX^{\prime\prime}$ be the set consisting of the AFs in $\AAnk[H]$  which are nontrivial on the rightmost block (nontrivial on all the blocks except possibly the leftmost and the rightmost). By using Proposition 7.1 in \cite{Tsiokos2} and Theorem \ref{general} we obtain:\begin{equation}\label{treesf2}\text{if $\F$ is nontrivial or $k>2$ then }\quad\tjhp(\WW_n)\rifff{H}{\XXX^\prime}\emptyset; \end{equation}  \begin{equation}\label{treesf3}\text{if $\F$ is trivial or $k>2$ then }\tjhpp(\WW_n)\rifff{H}{\XXX^{\prime\prime}}\emptyset. \end{equation} Make a choice of $\AAnk$-trees which justifies\footnote{That is, as in Definition \ref{lrif}.} the equations  (\ref{treesf}) (\ref{treesf2}) (\ref{treesf3}); and for each one of them, say $\Xi$, consider the tree\footnote{Here (in contrast to Definition \ref{pstr}) $\left(\prod_{i}\tj_i\right)\times\tjwa $ is defined in $H$ and maps each $(h_1,...,h_k)$ to $(j_1(h_1),...,j_{k_1}(h_{k-1}),\tjwa(h_k))$.} $\left(\left(\prod_{i}\tj_i\right)\times\tjwa\right)(\Xi)\circ(\F_1\times...\times\F_{k-1}\times p(w\F))$. Then by Exchange corollary \ref{exchange} (and proving\footnote{Use (ii) in Lemma \ref{Blem} an then continue with an $\AAnk$-path obtained by applying $\circ((\prod_i\F_i)\times p(w\F))$ to an embedding of an appropriate $(\F_{\emptyset,n}\rightarrow \WW_n,\BBnk[n])$-path.} containments in $\BBnk[GL_{N_1}\times... GL_{N_{i-1}}\times \Mmir{N}](\Di{D_{p(w\F)}}-\Di{D_{p(\wop\F)}})$) we directly obtain the Claim,  and we also obtain (\ref{ommomm}) as follows. If $k=2$ and $\F$ is trivial,  by successive uses of (\ref{treesf}) and (\ref{treesf3}) for $n$ replaced by $n,n-1,...$, (also when $n$ is replaced by $n-1,n-3,...$ we  interchange the right with the left), we obtain \begin{equation}\label{ptpa}\Omsd{(\tj_1\times\tjwa)(\WW_n)\circ(\F_1\times p(w\F))}\cap\Om=\Omsd{(\tj_1(\WW_{n,1}^{\deg}|^{\mathrm{gen}})\circ\F_1)\times  p(w(\tj(\WW_n^{\deg}|^{\mathrm{gen}})\circ\F)))} \end{equation}
where $\WW_{n,1}^{\deg}|^{\mathrm{gen}} $  (resp. $\WW_n^{\deg}|^{\mathrm{gen}}$) is the restriction of $\WW_{n,1}^{\deg}$ (resp. $\WW_n^{\deg}$) on the rows on which it is nontrivial. In case $k=2$ and $\F$ is nontrivial, we again obtain (\ref{ptpa}), by starting with a use of (\ref{treesf}) and (\ref{treesf2}) and then proceeding as in the previous sentence. From the last two sentences, from $(\tj_1(\WW_{n,1}^{\deg}|^{\mathrm{gen}})\circ\F_1)\times p(w(\WW_n^{\deg}|^{\mathrm{gen}}\circ\F)))\rif (\tj_1(\WW_{n,1}^{\deg})\circ\F_1)\times p(w(\WW_n^{\deg}\circ\F)))$ (which follows trivially), and since $\Omsd{\XX}\subseteq\Omm{\XX}$ for any $\XX$ (Main corollary \ref{maincor}), we obtain (\ref{ommomm}).\end{proof}
\begin{remark}The arguments in the proof above are similar to arguments found inside the proofs of Theorems 8.3.11, 8.3.12 and 8.3.18 in \cite{Tsiokos2}. For example---in relation to (\ref{treesf})---see Property 2 inside the proofs of these theorems\footnote{Note that the proof of Property 2.1 becomes simpler (and more similar) by replacing the $\eu$-step there with a $\co$-step.}. Also, in the special case that $\F$ is trivial, the proposition is implied by Theorems 8.3.11 and 8.3.12 in \cite{Tsiokos2} as follows.
	
	Let $N_{\mathrm{hor}}:=(k-1)n+n-1$, and $j:H\rightarrow GL_{N_{\mathrm{hor}}}$ be the embedding given by $$(g_1,...,g_k)\rightarrow\begin{pmatrix}
	g_1&&\\&\ddots&\\&&g_k
	\end{pmatrix}. $$ We adopt Definition 8.3.9 in \cite{Tsiokos2} for $l=k-1$ and $N=N_{\mathrm{hor}}$; therefore $j_{\FF_{n,k,k-1}}$ is a restriction of $j\cdot\tjh$. Recall from the appendix the information about $\X{...}$ given in Definition \ref{mtrox} and in Proposition-Definition \ref{AOF}. Notice that for every $J\in\X{\tjh(\WW_n)} $, we can find (and we fix) a $J_{N_{\mathrm{hor}}}\in\X{\F_{n,k,k-1}}$ which belongs to the Lie algebra of the $GL_{N_{\mathrm{hor}}}$-parabolic with Levi $j(H)$, and with projection to  $\Lie{j(H)}$ being equal to $j(J)$. We have $$\Di{D_{\tjh(\WW_n)}}+\Di{U_P}=\Di{D_{\F_{n,k,k-1}}}\leq\frac{\Di{\Om(J_{N_{\mathrm{hor}}})}}{2}\leq\frac{\Di{\Om(J)}}{2}+\Di{U_P}$$ where the first (resp. second) inequality follows for example from A in Main corollary \ref{maincor} (resp. 7.1.1 in \cite{Col}).  Therefore if $\Di{D_{\tjh(\WW_n)}}=\frac{\Di{\Om(J)}}{2}$ the inequalities above become equalities and  hence:\begin{itemize}\item By Main corollary \ref{maincor} we have $\Om(J_{N_{\mathrm{hor}}})\in\Omm{\F_{n,k,k-1}}$; the set $\Omm{\F_{n,k,k-1}}$  is calculated in Theorems 8.3.11 and 8.3.12 in \cite{Tsiokos2}.\item By using 7.1.1 in \cite{Col} we obtain that $\Om(J_{N_{\mathrm{hor}}})$ is induced from $\Om(J)$.\end{itemize} By using these theorems from \cite{Tsiokos2} and 7.2.5 in \cite{Col} we see how and which among the orbits in  $\Omm{\F_{n,k,k-1}}$ are induced from $\Lie{H}$ and hence we are done.  
\end{remark}
\begin{prop}\label{elem}Let $H:=GL_{n_1}\times... GL_{n_k}$ and $n:=\max\{n_1,...,n_k\}$. Let (for $1\leq i\leq k$) let $j_i$ be an embedding of $GL_{n_i} $ on $GL_n$ given by  $j_i(g)=\begin{pmatrix}I_{x_i}&&\\&g&\\&&I_{n-x_i-n_i}\end{pmatrix}$ for each $g\in GL_{n_i}$ (and some choice of $x_i$), and let $\tjh:=j_1^{-1}\times...\times j_k^{-1}$. Let $\F\in\AAA(U_n)$,  $[b_1,...]:=\Om(J_\F)$, and $a:=[a_{1,1},...]\times...\times[a_{1,k},...]\in\UUU{H}$ be such that there is an $(\tjh(\F)\rightarrow \Prink[H][a])$-path. Then $$\sum_i (b_i-1)\leq \sum_{i,j} (a_{i,j}-1). $$
\end{prop}\begin{tproof}Since $\sum_{i,j} (a_{i,j}-1)$ admits the smallest value for $a\in\Om(\tjh(\F))$ (and since this set does not depend on the choice of $\AAnk$-tree defining it), it is sufficient to fix any $(\tjh(\F)\rightarrow \Prink[H])$-tree and replace in the statement of the proposition the information `` there is an $(\tjh(\F)\rightarrow \Prink[H][a])$-path" with the information ``an output AF of the $\AAnk$-tree we just fixed belongs to $\Prink[H][a]$".   

The proof of this modification of the proposition for a certain choice of the  fixed $\AAnk$-tree follows inductively on $k$ from the claim below and from Proposition 7.1 (Part 1) in \cite{Tsiokos2}. This claim in turn easily follows from the proof of Theorem \ref{general}.
	
	\vsp
	\noindent\textbf{Claim. }\textit{Recall the meaning of $\sXi(...)$, which appears in the appendix in the proof of Theorem \ref{general}; here we identify $\mirsl{n_1}$ with the subgroup $\mirsl{n_1}\times 1\times...\times 1$ of $H$. Let $\I$ be an initial path of $\sXi(\tjh(\F))$ with output vertex an output vertex of $\sXi(\tjh(\F))$, and let $\Z$ be the label of this vertex. Let $w_1,...,w_y$ be the elements in $W_n$ such that in the $i$-th use of Case 3 in the proof of Theorem \ref{general} when we follow the path $\I$, we have $w=w_i$ (we include the cases in which $w$ is trivial). Consider the $\e$-steps of $\quasi{\I}$ which occur as in Case 2 and so that their label on the vertex they share with $\I$ is not their constant term. Let $c^\prime_1,...,c^\prime_y$ be the numbers such that the $i$-th such $\e$-step of $\quasi{\I}$ lies in the $c_i^\prime$-th column (of $n_1\times n_1$ matrices). Let $c_1,...,c_y$ be the numbers such that $(\prod_{1\leq j<i} w_j)j_1^{-1}(c_i)(\prod_{1\leq j<i} w_j)^{-1}=c_i^\prime$ where here $c_i$ (resp. $c_i^\prime$) is identified with the  torus, say $T$, in $\DDDD[n]$ (resp. $\DDDD[n_1]$) for which $\Set{T}=\{c_i\}$ (resp. $\Set{T}=\{c^\prime_i\}$). Then $\Z=\Z_1\times (\prod_{2\leq i\leq k}j_i^{-1})(\Z_2)$ where $\Z_1\in\AAnk(U_{n_1})$,\begin{equation}\label{restr}D_{\Z_2}\supseteq \prod_{i<j\not\in\{c_1,...c_y\}}U_{(i,j)}\end{equation} and $\F$ and $\Z_2$ have the same restriction on the group on the right hand side of (\ref{restr}).}
\end{tproof}
\begin{prop}\label{Sprop}
	Consider any $\eI\in\SSS$ with: $\sd(\eI)\leq 0 $; $\F_1,...,\F_{k-1},\F$ being trivial; $k>2$; and the elements of any among $\pi_1,...\pi_{k-1},$ and $\pi$, admitting absolutely convergent Eisenstein series expansion over discrete data\footnote{To be clear, we count the representations in the discrete spectrum as such representations. Also, in  \cite{Tsiokos2} we denoted the set of such representations by  $\Aut{\kkk,n,>}.$}. Then $I=0$.

\end{prop}\begin{tproof}It is known that:  \begin{Nilor}\label{nilorb}The dimension of an orbit $[a_1,...,a_z]\in\UUU{n}$ where $a_1\geq...\geq a_z$, is equal to two times the number of nondiagonal upper triangular entries (of $n\times n$ matrices) except for the entries in the $\sum_{1\leq i\leq j}a_i$-th row for $j=1,...,z$.\end{Nilor}

Since $\sd(\eI)\leq 0$, by the classification of automorphic forms, Corollary 9.3.5 in \cite{Tsiokos2}, and Nilpotent orbit dimension formula \ref{nilorb}, we obtain a trivial representation of a general linear group, say $GL_t$, which appears in the discrete data of one among $\pi_1$,...,$\pi_{k-1}$, and \wlo we assume this is the case for $\pi$. By Proposition \ref{prop1} we are also assuming \wlo that $GL_1$ does not appear in the inducing data of any among  $\pi_1,...,\pi_{k-1},\pi$. Hence, by replacing $E$ with an appropriate automorphic form admitting it as a residue (and by the way the dimension of a nilpotent orbit changes after applying induction), we are reduced to proving:

\vsp
\noindent\textbf{Claim.}\textit{ Let $\eI$ be as in the statement of the proposition except that: we replace ``$\sd(\eI)\leq 0$" with ``$\sd(\eI)\leq n-3$", we assume that $\pi$ is induced from $\Pmir{n}$ in exactly one way, that the representation on the $GL_{n-1}$-copy in the inducing data of this induction is nontrivial, that the elements in $\pi$ admit absolutely convergent Eisenstein series expansion over $\Pmir{n}$, and none among  $\pi_1,...,\pi_{k-1}$ are  induced  from $\Pmir{n}$. Then $I=0$.} 

\noindent{By} unfolding $I$ as usual (Eisenstein series expansion of $E$ over $\Pmir{n}$, and Fourier expansion of $\varphi$ over $U_n(\kkk)\s U_n(\A)$) and then using Proposition \ref{elem}, we see the Claim is obtained from: 

\vsp 
\noindent\textbf{Claim$^\prime$. }\textit{Let $a_1,..a_{k-1} $ be orbits in $\UUU{n}$, and $a_k$  be an orbit in $\UUU{n-1}$, and assume that none of them is the trivial or the minimal (that is, $[1,1,...]$ or $[2,1,1,...]$). Let $[a_{i,1},a_{i,2},...]:=a_i.$ Assume that $\sum_{i,j} (a_{i,j}-1)\geq n-1$. Then \begin{equation}\label{alen}\sum_{i}\frac{\Di{a_i}}{2}> \frac{n(n-1)}{2}+n-3.\end{equation} (of course, by the usual convention, $i$ varies over $1,...,k$.)} By using the Nilpotent orbit dimension formula \ref{nilorb} we easily obtain any of the claims below (at least some of which are of course familiar); among them, we only prove the last one, the other are even easier. Then Claim$^\prime$ is obtained as follows:\begin{itemize}
	\item[(i).] Nilpotent orbit dimension formula \ref{nilorb} and Claim 1 reduces it to the case, which we call Claim$^{\prime\prime}$,  in which: $a_{i,1}\geq 3$ for at most one $i$, and if such an $i$ exist, we have $a_{i,j}\geq 2$ for all $j$.
	\item[(ii).] Then by Claim 2 we are reduced to the modification of Claim$^{\prime\prime}$, which we call Claim$^{\prime\prime\prime}$, in which, in case the $i$ as in (i) exists: we have $i=k$, we replace ``$\sum_{i,j} (a_{i,j}-1)\geq n-1$" with  ``$\sum_{i,j} (a_{i,j}-1)\geq n-2$", and we substract from the right hand side of (\ref{alen}) the number $\min\{a_{k,1},a_{k,2},...\}-1$.
	\item[(iii).] Then by using Claim 3 we reduce Claim$^{\prime\prime\prime}$ to the special case which we call Claim 4. 
\end{itemize} 

\vsp 
\noindent\textbf{Claim 1 }(follows from \cite{Col}, Theorem 6.2.5).\textit{ If for two $a,b\in\UUU{n}$ we have $a>b$, then $\Di{a}>\Di{b}$. In particular, if $a$ is obtained from $b$ by changing two terms $r\geq r^\prime$ of the later one with the terms $r+1,r^\prime-1$, then $\Di{a}\geq\Di{b}$.}

\vsp
\noindent\textbf{Claim 2. }\textit{Let $[a_1,...,a_y]\in\UUU{n}$ be such $a_1\geq...\geq a_y\geq 2$ and $a_1>2$. Then $$\Di{[a_1,...,a_y]}+\Di{[2^l,1^{n-1-l}]}\geq \Di{[2^l,1^{n-l}]}+\Di{[a_1,...,a_{y-1},a_y-1]}+2(a_y-1).$$  }
		
\vsp 
\noindent\textbf{Claim 3. }\textit{Consider positive integers $l_1>l_2$ which are smaller from $\frac{n}{2}$. Then $$\Di{[2^{l_2},1^{n-l_2}]}+\Di{[2^{l_1},1^{n-l_1}]}\geq\Di{[2^{l_2-1},1^{n-l_2+1}]}+\Di{[2^{l_1+1},1^{n-l_1-1}]}.$$}		

\vsp 
\noindent\textbf{Claim 4. }Let $a=[a_1,...,a_y]\in\UUU{n-1}$ and $n\leq l\leq 2$ be such that: \begin{equation}\label{zkoual}2+l+\sum_i( a_i-1)\geq n-2;\end{equation}$a_1\geq...\geq a_y$;  if $a_1\geq 3$ then $a_y\geq 2$; and if $a_1=2$ then the left hand side of (\ref{zkoual}) is bigger or equal to $n-1$. If the left hand side of (\ref{zkoual}) is bigger or equal to $n-1$ we define $x:=0$; otherwise we define $x:=a_y-1$. We have: \textit{\begin{equation}\label{cl6}\frac{\Di{[2^2,1^{n-4}]}}{2}+\frac{\Di{[2^l,1^{n-l}]}}{2}+\frac{\Di{a}}{2}> \frac{n(n-1)}{2}-x+n-3.\end{equation}}

\vsp\noindent\textit{Proof of Claim 4.} We are easily reduced to assuming that the left hand side of (\ref{zkoual}) is equal to $n-1$ or $n-2$ (and hence $l$ is respectively equal to $y-2$ or $y-3$). Then, due to Nilpotent orbit dimension formula \ref{nilorb}, the left hand side of (\ref{cl6}) is equal to $\frac{n(n-1)}{2}-x^\prime+s_2+...s_{l+2}$ where: $s_2=a_1-2,$ $s_3:=a_1+a_2$, for $2<i\leq l+2$ we have   $s_i:=a_1+...+a_{i-1}-2i+6$, and if $x=0$ (resp. $x=a_y-1$) we define $x^\prime:=0$ (resp. $x^\prime:=a_y$). Indeed: consider the three triangles below, each of them containing the strictly upper triangular entries of $n\times n$ matrices (which are $\frac{n(n-1)}{2}$ in total), and the trapeziums on the third triangle with horizontal bases have heights (that is number of rows nontrivially intersecting them ) $a_1-1,...,a_y-1$ starting from the uppermost one and moving downwards; for $1\leq i\leq l+2$ let $z_i$ be the number of entries 
in the trapezium within which $z_i$ lies in one of the first two triangles; we then notice that each $z_i-s_{i}$ is also the number of entries of the area enclosed by the polygon containing this number (this time in the third triangle); then by using Nilpotent orbit dimension formula \ref{nilorb} we obtain $\frac{\Di{[2^2,1^{n-4}]}}{2}=z_1+z_2$, $\frac{\Di{[2^l,1^{n-l}]}}{2}=\sum_{3\leq i\leq l+2}z_i$, and $\frac{\Di{a}}{2}$ is equal to the number of entries of the trapeziums (with horizontal bases) on the third triangle. Hence we finish the proof by checking\footnote{After checking a few cases we are reduced to assuming: $a_1\geq 3$ and $y\geq 6$. Then $s_2+...+s_{l+2}\geq a_1-2+(a_1+a_2)(y-3)\geq +1+(a_1+a_2)\frac{y}{2}\geq 1+(a_1+...+a_y)\geq n$. } that $s_2+...+s_{l+2}>n-2$.

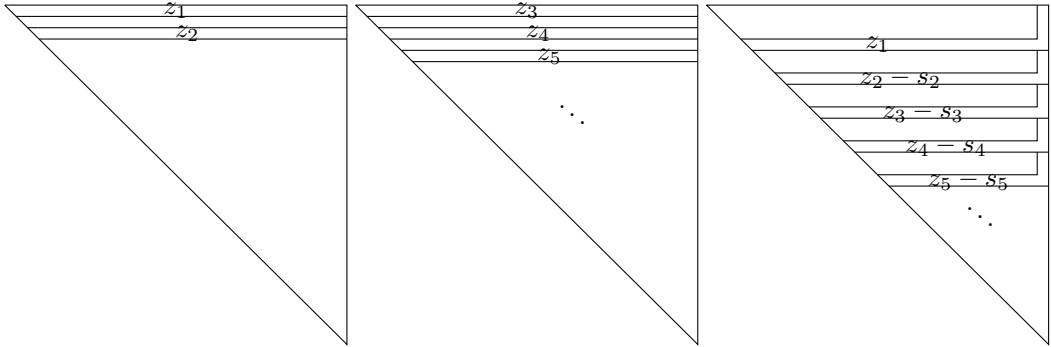
\newcommand{\tr}{30} 
\begin{tikzpicture}[scale=0.15]
\draw(0,\tr)--(\tr,\tr)--(\tr,0)--(0,\tr);
\draw(1,\tr-1)--(\tr,\tr-1);
\draw(2,\tr-2)--(\tr,\tr-2);
 \draw(3,\tr-3)--(\tr,\tr-3);
 \draw(\tr/2,\tr-0.5)node{$z_1$};
 \draw(\tr/2+1,\tr-2.5)node{$z_2$};
\end{tikzpicture}\hspace{1mm}\begin{tikzpicture}[scale=0.15]
\draw(0,\tr)--(\tr,\tr)--(\tr,0)--(0,\tr);
\draw(1,\tr-1)--(\tr,\tr-1);
\draw(2,\tr-2)--(\tr,\tr-2);
\draw(3,\tr-3)--(\tr,\tr-3);
\draw(4,\tr-4)--(\tr,\tr-4);
\draw(5,\tr-5)--(\tr,\tr-5);
\draw(\tr/2,\tr-0.5)node{$z_3$};
\draw(\tr/2+1,\tr-2.5)node{$z_4$};
\draw(\tr/2+2,\tr-4.5)node{$z_5$};
\draw(\tr-11,21)node{$\ddots$};
\end{tikzpicture}\hspace{1mm}\begin{tikzpicture}[scale=0.15]
\draw(0,\tr)--(\tr,\tr)--(\tr,0)--(0,\tr);
\draw(3,\tr-3)--(\tr-1,\tr-3)--(\tr-1,\tr);
\draw(4,\tr-4)--(\tr,\tr-4); 
\draw(6,\tr-6)--(\tr-1,\tr-6)--(\tr-1,\tr-4); 
\draw(7,\tr-7)--(\tr,\tr-7);
\draw(9,\tr-9)--(\tr-1,\tr-9)--(\tr-1,\tr-7);
\draw(10,\tr-10)--(\tr,\tr-10);
\draw(12,\tr-12)--(\tr-1,\tr-12)--(\tr-1,\tr-10);
\draw(13,\tr-13)--(\tr,\tr-13);
\draw(15,\tr-15)--(\tr-1,\tr-15)--(\tr-1,\tr-13);
\draw(16,\tr-16)--(\tr,\tr-16); 
\draw(\tr/2,\tr-3.5)node{$z_1$};
\draw(\tr/2+2,\tr-6.5)node{$z_2-s_2$};
\draw(\tr/2+4,\tr-9.5)node{$z_3-s_3$};
\draw(\tr/2+6,\tr-12.5)node{$z_4-s_4$};
\draw(\tr/2+8,\tr-15.5)node{$z_5-s_5$};
\draw(\tr/2+9,12)node{$\ddots$};
\end{tikzpicture}

\hspace{125mm}$\square$Claim 4\end{tproof}
\begin{remark}[Avoiding the use of Main corollary \ref{maincor}.]\label{avmain} In the beginning of the proof above, we made use of Proposition  \ref{prop1}, in the proof of which in turn we used Main corollary \ref{maincor}. We can avoid using Proposition \ref{prop1} by instead proving the modification of Claim$^\prime$ in which: the orbits $a_1,...,a_k$ are allowed to be equal to the minimal one, $a_k$ can even be the trivial one, and we remove $n-3$ in the right hand side of (\ref{alen}). This modification of  Claim$^\prime$ is obtained for example by repeating (i), (ii), and (iii), except that in (iii) we are reduced to a statement easier than Claim 4.\end{remark}
\begin{remark}[Relations of the present paper with the literature]\label{rlit}In the present paper we used slight refinements of results on Fourier coefficients appearing in \cite{Tsiokos2} (and we explain there how this paper in turn relates to the literature). 
	
	As for the unfolding of integrals until we are left with Fourier coefficients (that is, (\ref{sumprop})), except for the use of Lemma \ref{doubc}, we proceded as in the most familiar $GL_n\times GL_n$ Rankin-Selberg integrals\footnote{To be clear, the present paper contains no information about which $L$-functions appear.}. Proposition \ref{prop1} in the special case $k=2$ $N_1=N=n$ (and hence $\F_{1}=\F=\F_{\emptyset,n}$) is found in \cite{Ginzeul} (in Section 5). Proposition \ref{prop1} in the special case that $k=1$ and $D_{\F}$ is generated by root groups appears in \cite{Tsiokos2} (in 8.4). 
	
	Next we discuss the results about vanishing of integrals. Conjecture 1 in \cite{GinzAn} has as special case: a special case of Proposition \ref{prop1} in which $\F_1,...,\F_{k-1},\F$ are among certain AFs, say $\Z$, satisfying $\Om(\Z)=\{\Om(J_\Z)\}$. Proposition \ref{Sprop} is the restriction of Conjecture 1 in \cite{Ginzl} to representations with elements admitting absolutely convergent Eisenstein series expansion over discrete data. Since a proof of many cases of Conjecture 2 in \cite{Ginzl} appear in \cite{Ginzl}, and since I see no mention of my name, I point out that at least 4 months before the submission of \cite{Ginzl} (on the arXiv), the author of \cite{Ginzl} received emails from me which included most of the work\footnote{in the situation that one has available  the part of the literature preceding \cite{Ginzl} (and hence preceding \cite{Tsiokos2}).} for obtaining any of the three following proofs of Conjecture 2:
	
	\vsp
	\noindent\textit{Proof 1.} Use: Main corollary \ref{maincor} (and Exchange corollary \ref{exchange}), (\ref{treesf}), (\ref{treesf2}), and (\ref{treesf3}).\hfill$\square$
	
	\vsp 
	\noindent\textit{Proof 2.} Use: Main corollary \ref{maincor}, the claim inside the proof of Proposition \ref{elem} only for $I$ being an $\e$-path, Claim 1 and Claim 3 in the proof of Proposition \ref{Sprop}, and Nilpotent orbit dimension formula \ref{nilorb}.\hfill$\square$
	
	\vsp 
	\noindent\textit{Proof 3.} Change Proof 2 by: Removing the use of Main corollary \ref{maincor}, and using (this time more of the essence of) Proposition \ref{elem}.\hfill$\square$

 	\noindent{He} received these emails while I was a postdoc in Tel Aviv University and he was the person over there to whom I was presenting progress in my research. In particular, in 14 May 2016 I emailed him a document which contained: Theorem 3.1 of \cite{Tsiokos2} (both statement and proof);  most of the work of section 6 in \cite{Tsiokos2}. I mention Theorem 3.1 of \cite{Tsiokos2} because it overlaps with the proof in \cite{Ginzl} and because it leads to Proof 3 above, which is much shorter. I mention Section 6 in \cite{Tsiokos2} because together with Theorem 3.1 it gives even shorter proofs (Proofs 1 and 2 above).  Note that Proof 3 above remains much shorter from the proof in \cite{Ginzl} even after uncovering as needed the proof of Theorem 3.1 in \cite{Tsiokos2}. In the proof in \cite{Ginzl} the representations that are addressed are not addressed simultaneously.  	\end{remark}

\section{Appendix: general results on $\AAnk$-trees}Below we recall and consider simple extensions of results and definitions from \cite{Tsiokos2}. The differences in the proofs are minor. We present the proof of the first Theorem below in a self contained way, and we mostly skip the other proofs (which are even more identical to the ones in \cite{Tsiokos2}). 

Recall from \cite{Tsiokos2} that for a $\kkk$-group $V$, the set of AFs with domain $V$ is denoted by $\AAnk(V)$ (and hence the subset of $\kkk$-AFs with $\AAA(V)$). Also, for an algebraic group $H$, the set of AFs with domain contained in $H$ is denoted by $\AAnk[H]$.   \begin{theorem}[It extends 3.1 in \cite{Tsiokos2}]\label{general}Let $\mirsl{n}:=j(SL_{n-1})\prod_{1<j\leq n} U_{(1,j)}$ where $j$ is the lower right corner embedding of $SL_{n-1}$ in $GL_n$ (``as" stands for ``associate"). Let $\F$ be an AF over $\kkk$ such that there is a (linear) $\kkk$-group containing both $D_\F$ and $\mirsl{n}$ as $\kkk$-subgroups. We assume that $D_\F$ normalizes $\mirsl{n}$ and that there is a $\kkk$-homomorphism $f:\mirsl{n}D_\F\rightarrow \mirsl{n}$ such that its restriction to $\mirsl{n}$ is the identity morphism and $f(D_\F)\subseteq U_n$. Then there is an $$\left(\F\rightarrow\AAA(U_n)D_\F,\mirsl{n}D_\F,\kkk\right)\text{-tree}.$$
\end{theorem}\begin{tproof}We freely use notations from \cite{Tsiokos2}. Let $\XXX_n$ be the set of AFs $\F$ which are required to satisfy everything in the statement of the theorem except for the last sentence. For each $\F\in\XXX_n$ we inductively define an $\left(\F\rightarrow\AAA(U_n)D_\F,\mirsl{n}D_\F,\kkk\right)$-tree which we call $\sXi(\F).$ The induction is on $n$ and downwards on $\Di{D_\F\cap U_n}.$ More precisely: if $D_\F\supseteq U_n$ we define $\sXi(\F)$ to be the trivial $\F$-tree (that is, $\sXi(\F)$ has only one vertex); and otherwise we define $\sXi(\F)$, assuming that  $\sXi(\F^\prime)$ is defined for all $n^\prime $ and $\F^\prime\in\XXX_{n^\prime} $ satisfying one among 1) and 2) below $$1)\hspace{2mm}n^\prime<n\quad\quad\quad 2)\hspace{2mm}n^\prime=n\text{ and }\Di{D_{\F^\prime}\cap U_n}>\Di{D_\F\cap U_n}. $$  

Let $j$ be the lower right corner embedding of $GL_{n-1}$ in $GL_n.$ Assume that an AF $\F^\ih\in \XXX_n$ satisfies \begin{equation}\label{zsmir}\F^\ih= j(\F^{\prime})\circ\YY\end{equation}
where\begin{itemize}
	\item $\YY$ is an AF with domain $D_{\YY}=\prod_{1<j\leq n}U_{(1,j)},$ which is trivial on all root groups contained in $D_\YY$ except possibly $U_{(1,2)}$ (therefore $j(\mirsl{n-1}(\kkk))\in\Stab{\mirsl{n}}{\YY}$);
	\item $\F^\prime\in\XXX_{n-1}$,  and $j(\F^\prime)$ is defined  by extending $j$ so that its domain contains $D_{\F^\prime}$ and the third sentence in the statement of the theorem holds after replacing $D_\F$ with $j(D_{\F^\prime})$.\end{itemize} Then we define $\sXi(\F^\ih):=j(\sXi(\F^{\prime}))\circ\YY.$

We will define an $(\F,\mirsl{n}D_\F,\kkk)$-tree $\xi$, for each output AF ${\F^\ih}^\prime$ of which: we have $f(D_{{\F^\ih}^\prime})\subseteq U_n$, and either satisfies $\Di{D_{{\F^\ih}^\prime}}>\Di{D_\F}$ or is chosen as in (\ref{zsmir}) in the place of $\F^\ih$.  After that, the proof ends by defining $$\sXi(\F):=\sXi({\F^\ih}^\prime)\vee_{{\F^\ih}^\prime}\xi$$ where ${\F^\ih}^\prime $ varies over all output AFs of $\xi$. 

In the cases that  $\xi$ will turn out to contain an  $\e$-step, say $\xi_\e$, this is its last $\AAnk$-step,  because the dimension of the intersection of $U_n$  with the domain of the terms of $\xi_\e$ (that is, the output AFs of $\xi_\e$) is bigger from $\Di{D_\F\cap U_n }$.

Consider the smallest number $m\geq 2$ for which the set $A_m:=\prod_{m\leq i\leq n}U_{(1,i)}$ is contained in $D_\F$. To be clear, by $\prod_{n+1\leq i\leq n}U_{(1,i)}$ we mean the trivial subgroup of $GL_n.$

\vspace{1mm}
\noindent\textit{Case $1$: $\F(A_m)=\{0\}$ and $2<m$.} In this case $\xi$ is chosen to be the $(\F,\e)$-step over $U_{(1,m-1)}$.\hspace*{50mm}\hfill$\square$Case 1 

\vspace{1mm}
\noindent\textit{Case $2$: $m=2$ and $\F(A_2)=\{0\}$.}  In this case (\ref{zsmir}) is true for $[\F^\ih\leftarrowtriangle\F],$ and hence we choose $\xi$ to be the trivial $\F$-tree.\hfill$\square$Case 2

\vspace{1mm}
\noindent\textit{Case $3$: $\F(A_m)\neq\{0\}$.} We start with some definitions and observations, and return to defining $\xi$ in the next paragraph. Consider integers $1<l\leq n$ and $1\leq k<l$, and an algebraic subgroup $D$ of $U_nD_\F$. We define $L_{(k,l)}$ to be the biggest algebraic subgroup of $U_{n}D_\F$ such that $f(L_{(k,l)})$ is generated by all the root groups in $U_n$ except the $U_{(i,l)}$ with $k<i< l$. For example $L_{(l-1,l)}=U_{n}D_\F$. For $k>1$, the group $L_{(k-1,l)}$ is a normal subgroup of $L_{(k,l)}$, and hence $D\cap L_{(k-1,l)}$ is a normal subgroup of $D\cap L_{{k,l}}$. Since unipotent algebraic groups are connected,  the identity embedding 
\begin{equation}\label{embedding1}j_{D,k,l}:D\cap L_{(k-1,l)}\s D\cap L_{(k,l)}\rightarrow L_{(k-1,l)}\s L_{(k,l)},\end{equation} is either trivial or an isomorphism. 

Back to defining $\xi,$ let $l_1$ be the biggest number such that $\F(U_{(1,l_1)})\neq\{0\}$. The $\AAnk$-tree $\xi$ will either finish with an  $\e$-step  in at  most $l_1-2$ $\AAnk$-steps or it will consist of $l_1-1$ $\AAnk$-steps each one being a $\eu$-step or a $\co$-step.

Let $\F^0:=\F. $ Assume that for a number $i $ satisfying $1\leq i\leq l_1-2$, the first $i-1$ $\AAnk$-steps of $\xi$ have been defined, that none of them was an $\e$-step, and that the output AF of the $i-1$-th one has been given the name $\F^{i-1}$. To avoid confusion, if $i=1$, this means that no $\AAnk$-steps have been defined. The $i$-th $\AAnk$-step in $\xi$---and in case it is not an $\e$-step, its output AF $\F^i$---are defined as follows:\begin{enumerate}
	\item If $j_{D_{\F^{i-1}},l_1-i,l_1}$ is trivial and $U_{(1,l_1-i)}\not\subseteq D_{\F^{i-1}} $, the $i$-th $\AAnk$-step in $\xi$ is the $(\F^{i-1},\e)$-step  over $U_{(1,l_1-i)}$. Hence this is the last $\AAnk$-step in $\xi.$
	\item If $j_{D_{\F^{i-1}},l_1-i,l_1}$ is trivial and $U_{(1,l_1-i)}\subseteq D_{\F^{i-1}} $, there is an element $u\in U_{(l_1-i,l_1)}(\kkk)$, such that $u\F^{i-1}(U_{(1,l_1-i)})=\{0\}$. We choose the $i$-th $\AAnk$-step in $\xi$ to be the $(\F^{i-1}\rightarrow u\F^{i-1},\co)$-step and we define $\F^i:=u\F^{i-1}$.
	\item Finally assume that $j_{D_{\F^{i-1}},l_1-i,l_1}$ is an isomorphism. Let $\F^i$ be the AF defined by \begin{multline}D_{\F^i}=U_{(1,l_1-i)}(D_{\F^{i-1}}\cap L_{(l_1-i-1,l_1)}),\qquad\F^i(U_{(1,l_1-i)})=\{0\},\quad\text{and}\\\F^i(u)=\F^{i-1}(u)\quad\forall u\in D_{\F^{i-1}}\cap L_{(l_1-i-1,l_1)}.\end{multline} The $i$-th $\AAnk$-step in $\xi$ is the $(\F^{i-1}\rightarrow\F^{i},\eu)$-step.  
\end{enumerate}   

Assuming  that no $\e$-step was encountered, we have   
\begin{equation}\label{downreplaced}\prod_{1<j\leq n}U_{(1,j)}\subseteq D_{\F^{l_1-2}}\subseteq L_{(1,l_1)}, \end{equation}
and the last $\AAnk$-step of $\xi$ is the $\co$-step obtained from the minimal length element  $w\in W_{n}$ such that (\ref{zsmir}) is correct for $[\F^\ih\tlarrow w\F^{l_1-2}].$\hfill$\square$Case 3
\end{tproof}Unless otherwise specified, \textbf{in the rest of the appendix}, $$H:=\prod_{1\leq i\leq k}GL_{n_i}$$ (for some positive integers $k,n_1,...,n_k$).  
\begin{defi}[$\ossa$-groups]\label{ossagroup}Let $S$ be a finite set of root groups in $H$ such that for any roots $\alpha,\beta$ for which $U_\alpha$ and $U_\beta$ belong in $S$ we have: neither $\alpha+\beta$ nor $\alpha-\beta$ is a root, and $\alpha\neq -\beta$. Any at most one dimensional algebraic subgroup of  $\prod_{V\in S} V$ is called a $\ossa$-group.

We say that a $\ossa$-group, say $L$, is nontrivial on an entry, if and only if this entry is nondiagonal and a matrix in $L$ is notrivial on this entry.\end{defi} 
\begin{defi}[$\X{\F}$, $J_\F$]\label{mtrox}Let $\F\in\AAnk[H].$ We define $$\X{\F}:=\{(J_1,....,J_k)\in\Lie{H}: \sum_{1\leq i\leq k}\mathrm{tr}(J_iu_i)=\F(u_1,...,u_k)\qquad\forall(u_1,...,u_k)\in D_\F\}$$(where $\mathrm{tr}$ denotes the trace). Further assume that $D_\F$ is generated by root groups; then we denote by $J_\F$ the unique element in $\X{\F}\cap D_\F^t$ (where $D_\F^t:=\{(J_1,...,J_k): (J^t_1,...,J_k^t)\in D_\F\}$, and $t$ denotes transpose).\end{defi}
	
\begin{defi}[\text{$\Prink[H]$, $\Prink[H][a]$}]\label{pridefi}	 We define $$\Prink[H]:=\{\F\in\AAnk[H]:D_\F\text{ is a unipotent radical of an }H\text{-parabolic subgroup}\},$$ and for an $a\in\UUU{H}$, we also define $\Prink[H][a]=\{\F\in\Prink[H]:\Om(J_\F)=a\}$. 

We frequently write $\Prink[n]$ (resp. $\Prink[n][a]$,...) in place of $\Prink[GL_n]$ (resp. $\Prink[GL_n][a]$,...). \end{defi}

	\begin{defi}[$\mathcal{B}_H(k)$, $\mathcal{B}_H(k)$-paths]\label{defBBB} For two AFs $\F^\prime,\F\in\AAnk[H]$, consider an  $(\F^\prime\rightarrow\F,H)$-path called $\Xi$. For every  $\e$-step $\xi$ of $\quasi{\Xi}$, say the $i$-th one, consider an algebraic subgroup of $H$ which acts by conjugation freely and transitively on a subset, say $S_i$, of the variety consisting of the terms of this $\e$-step, so that the vertex of $\Xi$ which is also an output vertex of  $\xi$, is labeled with an AF in $S_i$ (here we do not identify any different vertices). 
		
		For every nonegative integer $k$ for which we can choose these actions so that $$k\geq\Di{D_\F}-\Di{D_{\F^\prime}}-\sum_i\Di{S_i} $$ we say that $\Xi$ is an $(\F^\prime\rightarrow \F,\BBnk[H](k))$-path or (just) a $\BBnk[H](k)$-path. We denote by $\BBnk[H](k)$ the set of AFs $\F$ admitting an $(\F_{\emptyset,H}\rightarrow \F,\BBnk[H](k))$-path. We frequently write $\BBnk[H]$ instead of $\BBnk[H](0)$, and (similarly to other notations) $\BBnk[n]$ instead of $\BBnk[GL_n].$ 	\end{defi}

\begin{pdefi}[$\Om(\F)$, $\mult{a}{\F}$, $\Omm{\F}$]\label{AOF} Let $\F\in\AAnk[H]$ and $\Xi$ be an $(\F\rightarrow\Prink[H],H)$-tree (known to exist by iterative uses of Theorem \ref{general}). We define $\Om(\F)$ to be the set consistsing of the minimal elements of $$\{a\in\UUU{H}:\text{An output AF of }\Xi\text{ belongs to }\Prink[H][a]\}.$$ 
	Let $a\in\Om(\F)$. If the  output vertices of $\Xi$ having label in $\Prink[H][a] $ are finitely many, we denote by $\mult{a}{\F}$ the number of such vertices. If these vertices are infinitely many we write $\mult{a}{\F}=\infty.$
	It turns out that $\Om(\F) $ is also the set of minimal elements in $\{a\in\UUU[H]:\X{\F}\cap a\neq\emptyset\}$; hence we obtain that $\Om(\F)$ does not depend on the choice of $\Xi$, and with a similar argument $\mult{a}{\F}$ is also independent from $\Xi$ (see 6.15 in \cite{Tsiokos2}).
	Finally we define \begin{equation*}\Omm{\F}:=\{a\in\Om(\F):\mult{a}{\F}<\infty\}.\qedhere \end{equation*}
	\end{pdefi}
\begin{maincor}[It extends 6.17 in \cite{Tsiokos2}; here we also define $\Omsd{\F}$]\label{maincor}Consider  $\F\in\AAnk[H]$ and a nonegative integer $k$ such that there is a $(\F_{\emptyset,H}\rightarrow\F,\in\BBnk[H](k))$-path. Consider an orbit $a\in\UUU{H}$, such that there is an  $(\F\rightarrow \Prink[H],H )$-tree with an output AF belonging in $\Prink[H][a]$ (equivalently $\X{\F}\cap a\neq\emptyset$).  Then:\begin{enumerate}[A.]
		\item $\frac{\Di{a}}{2}\geq\Di{D_\F}-k; $
		\item If \begin{equation}\label{pthap}\frac{\Di{a}}{2}=\Di{D_\F}-k, \end{equation} then: $\F\in\BBnk[H](k^\prime)$ if and only if $k^\prime\geq k$;  we define $\Omsd{\F}=\{b\in\Om(\F):\Di{b}=\Di{D_\F}-k\}$. If for no data $a,k$ as in the first two sentences of the present corollary  formula (\ref{pthap}) holds, we define $\Omsd{\F}:=\emptyset$.
		\item $\Omsd{\F}\subseteq\{b\in\Om(\F):\mult{b}{\F}=1\}$.
		\item Assume that $\F\in\BBnk[H]$. Then $$\left\{b\in\Om(\F):\Di{D_\F}=\frac{\Di{b}}{2}\right\}=\Omsd{\F}=\Omm{\F}=\{b\in\Om(\F):\mult{b}{\F}=1\}$$(of course the first equality follows from the definition of $\Omsd{\F}$).
	\end{enumerate}
\end{maincor}
\begin{tproof}The differences with the proof of Main corollary 6.17 in \cite{Tsiokos2} are small, the two that matter the most are:\begin{itemize}
		\item As already mentioned in the definition above, the existence of an $(\F\rightarrow\Prink[H],H)$-tree is obtained by iterative uses of Theorem \ref{general} (instead of a single use of Theorem 3.1 in \cite{Tsiokos2}).
		\item To obtain C, we again prove that  $\X{\F}\cap b$ is a connected variety for $b\in\Omsd{\F};$ in fact, $\X{\F}\cap b$ is  irreducible (again), and we obtain this irreducibility inductively by again using that $\X{F^{\prime\prime}}\cap b$ is an open nonempty subset of $\X{\F^{\prime}}\cap b$, where $\F^\prime$ and $\F^{\prime\prime}$ are any two successive labels in a path $\Xi$ as in Definition \ref{defBBB} (by making the smallest choice of $k$); but in contrast to the case $k=0$, to obtain this openness we use the information $b\in\Omsd{\F}$ (instead of only using that $b$ is a subvariety of $\Lie{H}$ nontrivially intersecting $\X{\F}$).   
	\end{itemize}
	\end{tproof}
\begin{defi}[$\rifff{H}{\XXX}$]\label{lrif}Let (more generally than in the rest of the appendix) $H$ be an algebraic subgroup of a product of general linear groups. Let $\F\in\AAnk[H], $ and $\XXX\subseteq \AAnk[H]$. Consider  an $(\F,H)$-tree $\Xi.$ Let $\VVV$ be a set consisting of  output vertices of $\Xi$, and for each $u\in \VVV$ let $\F_{u} $ be the label of $u$. For every  AF $\Z$ appearing as the label of an output vertex of $\Xi$ which is not a vertex in $\VVV$,  we assume  there are infinitely many output vertices of $\Xi$ with label equal to $\Z.$ We then write \begin{equation}\label{xtr}\F\rifff{H}{\XXX}\{\F_{u}:u\in \VVV\}\cap\XXX.\end{equation} To be clear, the only notation introduced in (\ref{xtr}) is ``$\rifff{H}{\XXX}$" (the right hand side is a set intersection as usual). In case $H$ is a direct product of general linear groups (resp. $\XXX=\AAnk[H]$), we remove $H$ (resp. $\XXX$) from this notation.
\end{defi}
\begin{Ecor}[It extends 8.3.5 in \cite{Tsiokos2}]\label{exchange}Let $k,k_i$ be nonegative integers,  $\F\in\BBnk[n](k)$, $\F_i\in\BBnk[n](k_i)$, and $\F_i^\prime\in\AAnk[n]$, where $i$ varies over the elements of a set, say $X$. Let $d:=\Di{D_\F}-k$. Assume that \begin{enumerate}\item there is an $a\in\Om(\F)$ with $\frac{\Di{a}}{2}=d$ 
		\item $\Di{D_{\F_i}}-k_i=d$,
		\item $\F\rif\{\F_i^\prime:i\in X\}$,
		\item $\F_i\rif\F_i^\prime$.
			\end{enumerate}Then $$\Omsd{\F}=\bigcup_i\{a\in\Om(\F_i):\Di{a}=d\}, $$ also for each $i$ the set $\{a\in\Om(\F_i):\Di{a}=d\}$ is equal to $\Omsd{\F_i}$ or to $\emptyset$. 
\end{Ecor}
\begin{proof}It follows directly from Main corollary \ref{maincor}.
\end{proof}

\end{document}